\documentclass[11pt]{amsart}
\usepackage{url,enumitem, amsopn}
\usepackage{xcolor}
\usepackage{cleveref}

\newcommand{\cal}{\mathcal}

\newcommand{\cv}{\mbox{\rm conv}}
\newcommand{\reals}{\mbox{$\mathbb R$}}
\newcommand{\comps}{\mbox{$\mathbb C$}}
\newcommand{\rats}{\mbox{$\mathbb Q$}}
\newcommand{\nats}{\mbox{$\mathbb N$}}
\newcommand{\ints}{\mbox{$\mathbb Z$}}

\newcommand{\pyr}{{\mathop{\mathrm{Pyr}}\nolimits}}

\newcommand{\m}{{\mathop{\mathbf{m}}\nolimits}}

\newcommand{\1}{{\mathop{\mathbf{1}}\nolimits}}
\newcommand{\id}{{\mathop{\mathrm{ID}}\nolimits}}

\newcommand{\comment}[1]{}

\usepackage{amsmath, amsfonts, amssymb, latexsym, stmaryrd}
\usepackage{graphicx} 

\def\squarebox#1{\hbox to #1{\hfill\vbox to #1{\vfill}}}
\def\qed{\hspace*{\fill}
        \vbox{\hrule\hbox{\vrule\squarebox{.667em}\vrule}\hrule}\smallskip}

\theoremstyle{plain}
\newtheorem{lemma}{Lemma}[section]
\newtheorem{theorem}[lemma]{Theorem}
\newtheorem{corollary}[lemma]{Corollary}
\newtheorem{proposition}[lemma]{Proposition}
\theoremstyle{definition}
\newtheorem{claim}[lemma]{Claim}
\newtheorem{observation}[lemma]{Observation}
\newtheorem{definition}[lemma]{Definition}

\newtheorem{question}[lemma]{Question}
\newtheorem{example}[lemma]{Example}

\newtheorem*{rmk*}{Remark}
\newtheorem*{rmks*}{Remarks}
\newtheorem*{conventions*}{Conventions}
\newtheorem*{convention*}{Convention}
\newtheorem*{example*}{Example}

\def\squareforqed{\hbox{\rlap{$\sqcap$}$\sqcup$}}
\def\qed{\ifmmode\squareforqed\else{\unskip\nobreak\hfil
\penalty50\hskip1em\null\nobreak\hfil\squareforqed
\parfillskip=0pt\finalhyphendemerits=0\endgraf}\fi}

\newlength{\tablength}
\newlength{\spacelength}
\newcommand{\tabstar}{\hspace*{\tablength}}
\newcommand{\spacestar}{\hspace*{\spacelength}}
\def\obeytabs{\catcode`\^^I=\active}
{\obeytabs\global\let^^I=\tabstar}
{\obeyspaces\global\let =\spacestar}
\newenvironment{display}{\begingroup\obeylines\obeyspaces\obeytabs}{\endgroup}
\newenvironment{prog}{\begin{display}\parskip0pt\sf}{\end{display}}

\author{Geir Agnarsson}
\address{Department of Mathematical Sciences
  \\ George Mason University \\ Fairfax, VA  22030}
\email{geir@math.gmu.edu}

\author{Jim Lawrence}
\address{Department of Mathematical Sciences
  \\ George Mason University \\ Fairfax, VA  22030}
\email{lawrence@gmu.edu}

\title{Minkowski ideals and rings}


\date{\today}

\begin{document}

\begin{abstract}
  \emph{Minkowski rings} are certain rings of simple functions on
  the Euclidean space $W = {\reals}^d$
  with multiplicative structure derived from Minkowski addition of convex
  polytopes.  When the ring is (finitely) generated by a set ${\cal{P}}$
  of indicator functions of $n$ polytopes then the ring can be presented
  as ${\comps}[x_1,\ldots,x_n]/I$ when viewed
  as a ${\comps}$-algebra, where $I$ is the ideal describing all the relations
  implied by identities among Minkowski sums of elements of ${\cal{P}}$.
  We discuss in detail
  the $1$-dimensional case, the $d$-dimensional box case and the affine
  Coxeter arrangement in ${\reals}^2$ where the convex sets are formed
  by closed half-planes with bounding lines making the regular triangular
  grid in ${\reals}^2$.
  We also consider, for a given polytope $P$, the Minkowski ring
  $M^\pm_F(P)$ of the collection ${\cal{F}}(P)$
  of the nonempty faces of $P$ and their multiplicative inverses.
  Finally we prove some general properties of identities
  in the Minkowski ring of ${\cal{F}}(P)$; in particular, we show that
  Minkowski rings behave well under Cartesian product, namely that
  $M^\pm_F(P\times Q)
  \cong M^{\pm}_F(P)\otimes M^{\pm}_F(Q)$
  as ${\comps}$-algebras where $P$ and $Q$ are polytopes.
\end{abstract}

\maketitle

\section{Introduction}
\label{sec:intro}

Minkowski rings capture the algebraic relations that convex polytopes and
their Minkowski sums have. The polytopes are here assumed to be embedded
in a given Euclidean space $W = {\reals}^d$. Recall that for nonempty
convex polytopes $P$ and $Q$ of $W$ their \emph{Minkowski sum}
is the set
$P + Q = \{\tilde{p} + \tilde{q} : \tilde{p}\in P, \tilde{q}\in Q\}$ which
is also a convex polytope of $W$. If ${\cal{S}}$ is the additive
group of functions on $W$ generated by the indicator functions
$[P]$ of convex polytopes $P$, then a  well-defined product $\cdot$
on ${\cal{S}}$ can be obtained by letting $[P]\cdot[Q] = [P+Q]$ for
every pair of nonempty convex polytopes $P$ and $Q$. In this way ${\cal{S}}$
becomes a commutative ring and its multiplicative identity is
${\1}_{\tilde{0}} = [\tilde{0}]$, the indicator function of the origin
$\tilde{0}\in W$. The structure of ${\cal{S}}$ as a ring has been
studied by several authors, e.~g.,
\cite{Groemer, McMullen, Lawrence, Jay-Klaus, Morelli} and the term
\emph{Minkowski ring} for such a commutative ring was coined
in~\cite{Lawrence, Jay-Klaus}. For general properties of Minkowski rings,
their relations to similar Adams-Minkowski rings, $\lambda$-rings that
stem from $K$-theory and relevant references, we refer to the
introduction of the recent paper of~\cite{A-Lawrence}.

In this paper we study some explicit Minkowski rings of some common
classes of convex polytopes, especially those classes that are closed
under taking Minkowski sums. This paper can be viewed as a self-contained
continuation and an extended addendum to~\cite{A-Lawrence}. 

The rest of the paper is organized as follows:

In Section~\ref{sec:defs-obs}
we review some useful properties of Minkowski rings
in general. In particular we review properties of Minkowski ideals of both
polynomial rings $R = {\comps}[x_1,\ldots,x_n]$ and Laurent polynomial
rings $R^{\pm} =  {\comps}[x_1,\ldots,x_n]^{\pm}
= {\comps}[x_1^{\pm 1},\ldots,x_n^{\pm 1}]$.

In Section~\ref{sec:simple} we compute explicitly the Minkowski ring
of a single convex polytope contained in $W=\reals^d$
and, when $d=1$, the Minkowski ring of a closed interval and its endvertices.
 This latter case will demonstrate that the structure
of the Minkowski ring determined by a polytope and its faces
is not entirely determined by the combinatorics of the polytope.
We conclude this section with a corollary describing the Minkowski ring
of $d$-dimensional boxes whose vertices have integer coordinates. 

In Section~\ref{sec:Coxeter} we discuss in detail the Minkowski ring
of convex polygons in the Euclidean plane ${\reals}^2$ that
are bounded by lines that are horizontal or slanted with slopes of
$\pm\sqrt{3}$ that form the regular triangular grid of ${\reals}^2$.
Such a discrete collection of lines can be viewed as a
Coxeter arrangement of hyperplanes in the plane.
It turns out that the Minkowski ring for
these bounded closed sets in ${\reals}^2$ is finitely generated, where
all the relations stem from products of binomials, differences of monomials;
see Theorem~\ref{thm:A}. Further, we study one type of the occurring
relations and determine exactly when
such a relation is a defining relation for the Minkowski ring,
see Corollary~\ref{cor:minimal} and Proposition~\ref{prp:indicator-iff}.

In Section~\ref{sec:Cartesian} we show that the Minkowski ring construction
of polytopes and their faces
behaves well under Cartesian products. Explicitly, we
show that the Minkowski ring of the Cartesian product of two polytopes
and all of its faces, when viewed as an algebra over the complex number
field ${\comps}$, is the tensor product of the Minkowski rings of each
of the polytopes in the Cartesian product over ${\comps}$, as stated
in the main result of the section, Theorem~\ref{thm:Cart-isom}.

Finally, in Section~\ref{sec:summary} we recap our main results
and present a few questions that are relevant to our results, for further
study.

\section{Definitions and observations}
\label{sec:defs-obs}

Any set of real valued functions on the real Euclidean space $W = {\reals}^d$
generates an additive group in the usual pointwise way by letting
$(f+g)(\tilde{x}) = f(\tilde{x}) + g(\tilde{x})$ for any $\tilde{x}\in W$.
In particular, if one starts with a set of characteristic functions,
or indicator functions, of a collection of subsets of $W$, then we obtain
a corresponding subgroup of the group ${\cal{S}}$ of all functionals on $W$.

{\sc Convention:} For a set $S\subseteq W$ its characteristic function,
or indicator function, $W\rightarrow \comps$ is denoted by ${\1}_S$ or $[S]$.

Note that the collection ${\cal{P}}(W)$ of closed convex polyhedra of $W$ is
closed under taking Minkowski sums. Hence, if $M({\cal{P}}(W))$ denotes
the additive subgroup generated by indicator functions of all the sets in
${\cal{P}}(W)$ then for any $A,B\in {\cal{P}}(W)$ the indicator
function $[A+B]$ of the Minkowski sum of $A$ and $B$ is also contained
in $M({\cal{P}}(W))$. Defining a multiplication by
$[A]\cdot [B] := [A+B]$ when $A, B \in \cal{P}$ and extending linearly
makes $M({\cal{P}}(W))$ a commutative ring
we call the {\em Minkowski ring} of all closed convex polyhedra of $W$.
It is a  ${\ints}$-algebra. The Minkowski ring $M({\cal{P}}(W))$
is large, to say the least, as it is generated by all closed convex polyhedra 
of $W$. As such one has a presentation
\[
M({\cal{P}}(W))
\cong {\ints}[x_P : P\subseteq W\mbox{ is a closed convex polyhedron}]/I
\]
as $\ints$-algebras, where $I$ is an ideal of
${\ints}[x_P : P\subseteq W\mbox{ is a closed convex polyhedron}]$ capturing all
$\ints$-algebraic relations among the sets.
Finitely generated ${\ints}$-sub-algebras of 
$M({\cal{P}}(W))$ have some convenient algebraic properties like
Noetherianity and such. Hence, it makes sense to consider smaller
subrings of $M({\cal{P}}(W))$, in particular
(a) those generated by indicator functions of collections of certain
types of polytopes that are closed under taking Minkowski sums, and
(b) those that are finitely generated.
\begin{definition}
\label{def:Minks-poly}
For any set ${\cal{D}}$ of closed convex polyhedra of $W$ let $M({\cal{D}})$
be the Minkowski ring generated by all the indicator functions of
 sets from ${\cal{D}}$.  The indicator function of any nonempty convex polytope
has an inverse in $M(\cal P(W))$, and, when $\cal D$ 
consists of nonempty convex polytopes,
$M^{\pm}({\cal{D}})$ will denote the Minkowski ring generated by all the
indicator functions of
sets from ${\cal{D}}$ and their multiplicative inverses.
In particular,
$M({\cal{P}}(W))$ is the Minkowski ring of all closed convex
polyhedra in $W$.

For a finite collection $\{P_1,\ldots,P_n\}$ of convex
polytopes of $W$ we write $M(P_1,\ldots,P_n)$ for the ring
$M(\{P_1,\ldots,P_n\})$ and $M^{\pm}(P_1,\ldots,P_n)$ for the ring
$M^{\pm}(\{P_1,\ldots,P_n\})$.
\end{definition}
\begin{example}
\label{exa:closed}
Note that
if $A = \{(x,y)\in {\reals}^2 : x > 0 \mbox{ and } y\geq 1/x\}$
and $B = \{(x,0)\in {\reals}^2 : x\in {\reals}\}$; the x-axis,
then both $A$ and $B$ are closed sets but their Minkowski sum
$A+ B$ is not closed in ${\reals}^2$. Hence, we cannot equip
a ring structure to the collection of all closed sets of $W$ in
the same way that is done for ${\cal{P}}(W)$.

Also observe that convexity plays an important role in the structure
of Minkowski rings.  The Minkowski sum generalizes to a sum on
arbitrary sets in a vector space by the same formal definition, $A+B = \{a+b : a\in A, b\in B\}.$
Although the indicator function $F$ of the union
of two disjoint compact intervals of real numbers is in the Minkowski
ring of convex polytopes, its square in the Minkowki ring is not equal
to the indicator function of the (generalized) sum of this set with
itself (as is the case for indicator functions of convex polytopes).
\end{example}
Note that we clearly have that
\[
M(P_1,\ldots,P_n) = {\ints}[[P_1],\ldots,[P_n]] =
{\ints}[{\1}_{P_1},\ldots,{\1}_{P_n}].
\]
as ${\ints}$-algebras. As we have a natural $\ints$-algebra surjection 
$\phi : {\ints}[x_1,\ldots,x_n]\twoheadrightarrow M(P_1,\ldots,P_n)$ given
$\phi(x_i) = {\1}_{P_i}$ 
for each $i$ and extended naturally, we obtain
$M(P_1,\ldots,P_n)\cong {\ints}[x_1,\ldots,x_n]/I$ for an ideal
$I = \ker(\phi) \subseteq {\ints}[x_1,\ldots,x_n]$ and so,
by tensoring with $\comps$, we obtain 
\[
M(P_1,\ldots,P_n)\otimes_{\ints}{\comps}
\cong {\comps}[x_1,\ldots,x_n]/{\overline{I}}
\]
where $\overline{I} = I\otimes_{\ints}{\comps}$ is the corresponding extended
ideal. Because of this, we will view $M(P_1,\ldots,P_n)$ as a
$\comps$-algebra
and the ideal $I = \ker(\phi)$ as its extended ideal $\overline{I}$
and as a vector space
over $\comps$ rather than just a $\ints$-module. The prime ideal structure
of such rings was studied in~\cite{Jay-Klaus}.

When working with the Laurent polynomial ring
$R^{\pm}= \comps[x_1,\ldots,x_n]^{\pm} = \comps[x_1^{\pm 1},\ldots,x_n^{\pm 1}]$
in the variables $x_1,\ldots,x_n$ and their multiplicative inverses
$x_1^{-1},\ldots,x_n^{-1}$, we obtain the corresponding isomorphism
$M^{\pm}(P_1,\ldots,P_n)\cong {\comps}[x_1,\ldots,x_n]^{\pm}/I'$ where
$I'$ is obtained from $I$ by removing all generators consisting of monomials
in $x_1,\ldots,x_n$. When there is no danger of ambiguity we will denote
the ideal $I'$ by $I$ as for the polynomial ring. It should be clear
by the context whether we are dealing with ideals of the polynomial ring $R$
or the Laurent polynomial ring $R^{\pm}$.

{\sc Convention:} From now on we will view our Minkowski rings
$M(P_1,\ldots,P_n)$ and their corresponding Laurent rings
$M^{\pm}(P_1,\ldots,P_n)$ as algebras of the complex number field
${\comps}$.
\begin{definition}
\label{def:Minks-ideal}
An ideal $I$ of $R = \comps[x_1,\ldots,x_n]$ (resp.
$R^{\pm} = {\comps}[x_1,\ldots,x_n]^{\pm}$) is a {\em Minkowski ideal}
if there are polytopes $P_1,\ldots,P_n$ such that $I = \ker(\phi)$ 
where $\phi : \comps[x_1,\ldots,x_n] \twoheadrightarrow M(P_1,\ldots,P_n)$
(resp.~$\phi : \comps[x_1,\ldots,x_n]^{\pm}
\twoheadrightarrow M^{\pm}(P_1,\ldots,P_n)$)
is the aforementioned subjective $\comps$-algebra homomorphism onto
the Minkowski ring $M(P_1,\ldots,P_n)$ (resp. $M^{\pm}(P_1,\ldots,P_n)$).
\end{definition}

For a polynomial
$f = f(\tilde{x}) = f(x_1,\ldots,x_d)\in R = \comps[x_1,\ldots,x_d]$
let $f^{(i)}(\tilde{x}) = f(x_1^i,\ldots,x_d^i)$. Call an ideal $I$ of $R$
{\em power closed} if $f\in I$ implies $f^{(i)}\in I$ for each $i\in\nats$
(see~\cite{A-Lawrence}.) For any nonzero $k\in\reals$
the map $\lambda_K : W \rightarrow W$ defined
by $\lambda_k(C) = kC$ induces an isomorphism
$\tilde{\lambda}_k : M({\cal{P}}(W))\rightarrow M({\cal{P}}(W))$.
When $k\in\nats$
the restriction of $\lambda_k$ to 
any finitely generated $\comps$-algebra $M(P_1,\ldots,P_n)$ 
induces a well defined $\comps$-algebra homomorphism
$M(P_1,\ldots,P_n)\rightarrow M(P_1,\ldots,P_n)$. Viewing
$M(P_1,\ldots,P_n) \cong R/I$ we therefore have the following observation:
\begin{observation}
\label{obs:power-closed}
Any Minkowski ideal of $R=\comps[x_1,\ldots,x_n]$ is power closed.
\end{observation}

Similarly, for the Laurent polynomial ring
$R^{\pm}= \comps[x_1,\ldots,x_d]^{\pm} = \comps[x_1^{\pm 1},\ldots,x_d^{\pm 1}]$
in the variables $x_1,\ldots,x_n$ and their multiplicative inverses
$x_1^{-1},\ldots,x_n^{-1}$ we have the following~\cite{A-Lawrence}.
\begin{observation}
\label{obs:power-closed-Laurant}
Any Minkowski ideal of $R^{\pm} =\comps[x_1,\ldots,x_n]^{\pm}$ is power closed.
In fact, if $I\subseteq R^{\pm}$ is a Minkowski ideal, then $f\in I$
implies $f^{(i)}\in I$ for each $i\in\ints$. 
\end{observation}
In~\cite[Theorem 5.11]{A-Lawrence} power closed principal ideals
$R=\comps[x_1,\ldots,x_n]$ and $R^{\pm} =\comps[x_1,\ldots,x_n]^{\pm}$
are characterized. From this one can see that for $n=1$ not all power closed
ideals are Minkowski ideals (see~\cite{A-Lawrence}.)

In this paper we are particularly interested in the Minkowski rings of
the collections of closed sets
that are formed by a given collection of polytopes
$P_1,\ldots,P_n$ and all their faces, the arbitrary Minkowski sums of which
 forming a nice closed class of polytopes.

 \section{A few simple cases}
\label{sec:simple}

In this section 
we dispatch some explicit simple cases.
We first consider the Minkowski ring $M(S)$ for a single 
arbitrary closed convex polyhedron $S\subseteq W = {\reals}^d$,
in particular that
when $S = P$, a convex polytope in $W$. Secondly we consider the Minkowski
ring generated by a compact interval of $\reals$ together
with its endpoints.

\subsection{The principal case}
Here we consider the ``principal case'': the Minkowski ring $M(S)$ where
$S$ is a closed convex set. We first look at the case
$S=\emptyset$ and then the case $S\neq\emptyset$ when $d=1$, which will yield
the structure of $M(S)$ for $\emptyset\neq S\subseteq W$ in general.
\begin{example}
\label{exa:s-empty}  
If $S=\emptyset$, then for the surjection
$\phi : {\comps}[x]\twoheadrightarrow M(S) = M(\emptyset)$
given by $\phi(x) = {\1}_{\emptyset} = 0$, we have for $f\in \comps[x]$ that
$\phi(f) = 0$ iff $f(0){\1}_{\{0\}} = f(0){\1}_{M(\emptyset)} = 0$,
that is $f(0) = 0$ in $\comps$. This means exactly that $f\in (x)$,
and hence $\ker(\phi) = (x)$, a monomial ideal generated by $x$.
\end{example}
\begin{example}
  \label{exa:S-zero-elt}
Similarly, if $S = \{0\} = {\1}_{M(\{0\})}$, then for the surjection
$\phi : {\comps}[x]\twoheadrightarrow M(S) = M(\{0\})$
given by $\phi(x) = {\1}_{\{0\}}$, we have for $f\in \comps[x]$ that
$\phi(f) = 0$ iff $f(1){\1}_{\{0\}} = f(1){\1}_{M(\emptyset)} = 0$,
that is $f(1) = 0$ in $\comps$. This means exactly that $f\in (x-1)$,
and hence $\ker(\phi) = (x-1)$, a binomial ideal generated by $x-1$.
\end{example}
We now consider the case when $d=1$. Here a closed
convex set $S$ is a closed nonempty interval of $W = \reals$
and therefore $S$ has one of the following forms:
$S = [a,b]$, where $a\leq b$ are real numbers,
$S = [a,\infty[$,
$S = ]-\infty,b]$, or
$S = ]-\infty, \infty[$. 

Consider the case $S = [a,b]$ where $0<a<b$. The surjection 
$\phi : {\comps}[x]\twoheadrightarrow M(S) = M([a,b])$
given by $\phi(x) = {\1}_{[a,b]}$ satisfies $\phi(x^k) = {\1}_{[ka,kb]}$ for
each integer $k\geq 0$. If there is a nonzero
$f = \sum_{i=0}^Na_ix^i \in \comps[x]$ with $\phi(f) = 0$, then
there is one of least degree $N$. Since 
$\phi(f) = \sum_{i=0}^Na_i{\1}_{[ia,ib]} = 0$ in $M([a,b])$, then
if $\max(Na,(N-1)b) < t < Nb$ we obtain
\[
0 = \left(\sum_{i=0}^Na_i{\1}_{[ia,ib]}\right)(t) = a_N,
\]
contradicting the minimality of the degree of $f$. Hence $f=0$ must
hold and we have $\ker(\phi) = \{0\}$ in this case.
All these cases for which either $a$ or $b$ is nonzero are handled
similarly and for each of these cases we have $\ker(\phi) = \{0\}$.

For the remaining cases, $S$ is a closed interval such that
$tS = S$ for any real $t>0$ and $S\neq \{0\}$, equivalently $2S = S$ and
$S\neq\{0\}$. In this 
case we have for $f\in\comps$ that $\phi(f) = 0$ iff $f(0) = f(1)  = 0$
and so $f\in x(x-1)$. Hence $\ker(\phi) = (x(x-1))$; a principal ideal
generated by a product of a monomial and binomial.

Finally, consider $S\subseteq W = {\reals}^d$. We can use the above
to treat this case in the same way by restricting our attention
to $S\cap L(\tilde{a})$ where $L(\tilde{a}) = \{ t\tilde{a} : t\in\reals\}$
is the one dimensional vector space generated by $\tilde{a}$. Note
that (a) for any sets $A,B\subseteq W$ we have $t(A\cap B) = tA\cap tB$
and that (b) $\{t\in\reals : t\tilde{a} \in S\cap L(\tilde{a})\}$ is a closed
interval of $\reals$. Using these snippets we summarize this in the following.
\begin{observation}
\label{obs:principal}
Let $S\subseteq W = {\reals}^d$ be a closed set. For the surjection
$\phi : {\comps}[x]\twoheadrightarrow M(S)\subseteq M({\cal{P}}(W))$
$\phi(x) = S$ we have $I = \ker(\phi)$ is as follows:

If $S = \emptyset$ then $I = (x)$.

If $S = \{0\}$, then $I = (x-1)$.

If $S \not\in\{\emptyset, \{0\}\}$ and $S = 2S$, in particular for $S = W$,
then $I = (x(x-1))$.

If $S \not\in\{\emptyset, \{0\}\}$ and $S\neq 2S$, in particular for a bounded
convex polytope $S = P$, then $I = (0)$.
\end{observation} 

Allowing each of the generators $x_i$ to possess a multiplicative inverse,
we obtain the following corollary for the Laurent polynomial ring
${\comps}[x]^{\pm} = {\comps}[x,x^{-1}]$.
\begin{corollary}
\label{cor:principal-Laurent}
Let $S\subseteq W = {\reals}^d$ be a closed set. For the surjection
$\phi : {\comps}[x]^{\pm}\twoheadrightarrow
M^{\pm}(S)\subseteq M^{\pm}({\cal{P}}(W))$
$\phi(x) = S$ we have $I = \ker(\phi)$ is as follows:

If $S = \emptyset$ then $I = R^{\pm}$.

If $S = \{0\}$ or if $S \not\in\{\emptyset, \{0\}\}$ and $S = 2S$,
in particular for $S = W$, then $I = (x-1)$.

If $S \not\in\{\emptyset, \{0\}\}$ and $S\neq 2S$, in particular for a bounded
convex polytope $S = P$, then $I = (0)$.
\end{corollary} 

\subsection{A one dimensional case}

In this subsection we consider the Minkowski subring of
$M({\cal{P}}(\reals))$
generated by a bounded closed convex set of $\reals$, i.e.~a closed interval
$[\alpha,\beta]$
with real endpoints $\alpha,\beta\in\reals$. We will
consider the ring
$R = R_{\alpha,\beta} := M(\{\alpha\},\{\beta\},[\alpha,\beta])$ and
the surjection
$\phi : {\comps}[x,y,z] \twoheadrightarrow R_{\alpha,\beta}$ given by
$\phi(x) = {\1}_{\{\alpha\}}$, $\phi(y) = {\1}_{\{\beta\}}$ and
$\phi(z) = {\1}_{[\alpha,\beta]}$.

When $\alpha = 0$ and $\beta = 1$ we obtain one example of such a
Minkowski ring is $R_{0,1} = M(\{1\}, [0,1])$, the subring of
$M({\cal{P}}(\reals))$ generated by the closed interval
$[0,1]$ and its subfaces $\{0\}$, yielding the unit of the ring, and
$\{1\}$.

In general for $R_{\alpha,\beta} = M(\{\alpha\}, \{\beta\}, [\alpha,\beta])$
we see that the surjection $\phi$ satisfies
$\phi((z - x)(z - y)) = 0$ by direct computation, or
as indicated by~\cite[Lemma 2, p.~10]{Lawrence} which is
a special case of relation from~\cite[Theorem 5, p.~11]{Lawrence}
that always holds in the Minkowski ring.
Namely, $\prod_{v\in P}([P] - [v]) = 0$ always holds
in the Minkowski ring generated by a polytope $P$ and its vertices
$v$. Hence, we always have $(z-x)(z-y)\in\ker(\phi)$ for
any $\alpha,\beta\in {\reals}$. Whether or not
$\ker(\phi) = ((z-x)(z-y))$ depends on $\alpha$ and $\beta$.
Since $z^2 = (x+y)z -xy$ modulo
$\ker(\phi)$, we can view this relation as a reduction in
a Gr\"{o}bner Basis for the ideal $\ker(\phi)$ w.r.t.~the lexicographical
term order (LEX) where $x\prec y\prec z$ (or equivalently the
degree lexicographical order (DEGLEX)). Hence, any element
in ${\comps}[x,y,z]/\ker(\phi)\cong R_{\alpha,\beta}$
can be represented as $f(x,y)z + g(x,y)$ for some polynomials
$f$ and $g$ in two variables over ${\comps}$,
where each can, in some cases, be reduced further
as we will discuss here below.

\subsubsection{(A) $0 = \alpha < \beta$:}
\label{ssec:0b}
In this case $\phi(x) = {\1}_{\{0\}}$; the unit, and so 
$R_{0,\beta} = M(\{\beta\},[0,\beta])$ is a homomorphic image of
${\comps}[y,z]$ and
$(z-x)(z-y) = (z-1)(z-y)\in\ker(\phi)\subseteq {\comps}[y,z]$.
As a result any reduced element in ${\comps}[y,z]/\ker(\phi)\cong R_{0,\beta}$
has a representation as $f(y)z + g(y)$ where $f$ and $g$ are polynomials
in one variable over ${\comps}$. Assume that $f(y)z + g(y)\in \ker(\phi)$
where $f(t) = \sum_{i=1}^ka_it^i$ and $g(t) = \sum_{i=1}^hb_it^i$. By
definition we then obtain
\[
  {\1}_{\emptyset} = 
  \sum_{i=1}^ka_i[[i\beta,(i+1)\beta]] + \sum_{i=1}^hb_i[\{i\beta\}]  = 
  \sum_{i=1}^ka_i{\1}_{[i\beta,(i+1)\beta]} + \sum_{i=1}^hb_i{\1}_{\{i\beta\}}.
\]
Evaluating at $t = (i + 1/2)\beta$ for $i = 0,1,2,\ldots$ we obtain
$a_i = 0$ for each $i$ and then evaluating at each $t = i\beta$ we obtain
$b_i = 0$ for each $i$. This shows that $\ker(\phi) = ((z-1)(z-y))$.
Relabeling the variables for this case we have in this case the following
\begin{equation}
\label{eqn:0ab}
R_{0,\beta} = M(\{\beta\},[0,\beta]) \cong {\comps}[x,y]/((y-1)(y-x)).
\end{equation}
The case where $\beta <0$ is handled in the same way and we have
(\ref{eqn:0ab}) also when $\beta < 0$. 

\subsubsection{(B) $0 < \alpha < \beta$ and $\alpha/\beta\not\in {\rats}$:}
Here the surjection
$\phi : {\comps}[x,y,z] \twoheadrightarrow R_{\alpha,\beta}$ is given by
$\phi(x) = {\1}_{\{\alpha\}}$, $\phi(y) = {\1}_{\{\beta\}}$ and
$\phi(z) = {\1}_{[\alpha,\beta]}$. As mentioned above each element has
a reduced form $f(x,y)z + g(x,y)$ for some polynomials
$f = \sum_{i,j}a_{i\/j}x^iy^j$ and $g = \sum_{i,j}b_{i\/j}x^iy^j$
in two variables over ${\comps}$. Assume that $f = 0$ and that
$g(x,y) = \sum_{i,j}b_{i\/j}x^iy^j\in \ker(\phi)$. Again by mere definition
we obtain that
\[
  {\1}_{\emptyset}
  = \sum_{i,j}b_{i\/j}[[i\alpha + j\beta]]
  = \sum_{i,j}b_{i\/j}{\1}_{[i\alpha + j\beta]}.
\]
Since $\alpha/\beta\not\in{\rats}$, no two distinct linear combinations
$i\alpha + j\beta$ are identical as real numbers
and hence evaluating the above characteristic
function at each of these unique linear combinations, we get that
$b_{i\/j} = 0$ for each $i,j$ and hence $g(x,y) = 0$ in
${\comps}[x,y]$.

Assume now that $f(x,y)z + g(x,y)\in \ker(\phi)$. By
definition we obtain that
\begin{eqnarray*}
  {\1}_{\emptyset}
  & = & \sum_{i,j}a_{i\/j}[[(i+1)\alpha + j\beta, i\alpha + (j+1)\beta]]
  + \sum_{i,j}b_{i\/j}[\{i\alpha+ j\beta\}] \\
  & = & \sum_{i,j}a_{i\/j}{\1}_{[(i+1)\alpha + j\beta, i\alpha + (j+1)\beta]}
  + \sum_{i,j}b_{i\/j}{\1}_{\{i\alpha+ j\beta\}}.
\end{eqnarray*}
Note that each interval $[(i+1)\alpha + j\beta, i\alpha + (j+1)\beta]$
has the fixed length of $\beta - \alpha$.
Since $\alpha/\beta\not\in{\rats}$ there is a unique maximum
linear combination $\mu := i^*\alpha + (j^*+1)\beta$ and hence there exists
an element $m^* \in ](i^*+1)\alpha + j^*\beta, i^*\alpha + (j^*+1)\beta[$
in the corresponding open interval that is not contained in any other
interval $[(i+1)\alpha + j\beta, i\alpha + (j+1)\beta]$ and so there is an
open neighborhood $]m^* - \epsilon, m^* + \epsilon[$ completely contained
in $](i^*+1)\alpha + j^*\beta, i^*\alpha + (j^*+1)\beta[$ with an empty
intersection with all the other intervals
$[(i+1)\alpha + j\beta, i\alpha + (j+1)\beta]$. In this case we can pick
a real $\gamma \in ]m^* - \epsilon, m^* + \epsilon[$ that is not equal to
any linear combination $i\alpha + j\beta$. Evaluating the above
characteristic function at $\gamma$ we obtain
$0 = {\1}_{\emptyset}(\gamma) = a_{i^*\/j^*}$. By induction we obtain
that $a_{i\/j} = 0$ for each $i$ and $j$ and so $f(x,y) = 0$ in
${\comps}[x,y]$. Consequently we have then by the above argument that
$g(x,y) = 0$ in ${\comps}[x,y]$. 
This shows that $\ker(\phi) = ((z-x)(z-y))$ and so we have the following
\begin{equation}
\label{eqn:ab-notrat}
R_{\alpha,\beta} = M(\{\alpha\},\{\beta\},[\alpha,\beta])
\cong {\comps}[x,y,z]/((z-x)(z-y)).
\end{equation}
The case where $\alpha$ or $\beta$ are possibly negative are similar
and so we have (\ref{eqn:ab-notrat}) whenever neither $\alpha$ nor $\beta$
are zero and $\alpha/\beta\not\in{\rats}$.

\subsubsection{(C) $0 < \alpha < \beta$ and $\alpha/\beta\in {\rats}$:}
As in the previous case (B) we have $(z-x)(z-y)\in\ker(\phi)$ and
so any element
in ${\comps}[x,y,z]/\ker(\phi)\cong R_{\alpha,\beta}$
can be represented as $f(x,y)z + g(x,y)$ for some polynomials $f$ and $g$
in two variables over ${\comps}$. In this case we have moreover that
$\alpha/\beta\in{\rats}$ and so there are relatively prime
$m,n\in\nats$ with $\alpha/\beta = m/n$ and hence $y^m - x^n \in\ker(\phi)$
as well. We now argue that $\ker(\phi) = ((z-x)(z-x),y^m - x^n)$.

Suppose first that
$g(x,y) = \sum_{i,j}b_{i\/j}x^iy^j\in\ker(\phi)$. We can partition
the index set $I_g = \{(i,j) : b_{i\/j}\neq 0\}$ by the distinct values
of $i\alpha + j\beta$, that is $I_g = I_1\cup\cdots\cup I_p$ where each
$I_k = \{ (i,j)\in I_g : i\alpha + j\beta = c_k\}$ where $c_k$ is some
real constant. Hence $g(x,y) = \sum_k g_k(x,y)$ where each
$g_k(x,y) = \sum_{(i,j)\in I_k} b_{i\/j}x^iy^j$. By assumption we have
\[
  {\1}_{\emptyset}
  = \sum_{i,j}b_{i\/j}[[i\alpha + j\beta]]
  = \sum_{i,j}b_{i\/j}{\1}_{[i\alpha + j\beta]}.
\]
By evaluating this characteristic function at $c_k$ for $k = 1,\ldots,p$
we get that $ \sum_{(i,j)\in I_k} b_{i\/j} = 0$ for each $k$ and this condition
is sufficient for $g(x,y)\in\ker(\phi)$. If $(i,j), (i',j')\in I_k$, then
$i\alpha + j\beta = c_k = i'\alpha + j'\beta$ and so
$(i',j') = (i,j) + {\ell}(n,-m)$ for some ${\ell}\in{\ints}$. By induction
on $|I_k|$ this implies that each $g_k(x,y)$ is a linear combination of
binomials of the form
\begin{equation}
\label{eqn:m-n}  
x^iy^j - x^{i+\ell n}y^{j-\ell m} = x^iy^{j-\ell m}(y^{\ell m} - x^{\ell n})
\in ((y^m - x^n)).
\end{equation}
This implies that $g(x,y) = \sum_k g_k(x,y)\in ((y^m - x^n))$.

Note that $\{z^2 - (x+y)z + xy, y^m - x^n\}$ forms a Gr\"{o}bner Basis
for the ideal $((z-x)(z-x),y^m - x^n)$ of ${\comps}[x,y,z]$ w.r.t.~LEX
in which $x\preceq y\preceq z$
and so each element in $R_{\alpha,\beta}$ has a representation of
\[
f(x,y)z + g(x,y) = \left(\sum_{j=0}^{m-1}f_j(x)y^j\right)z
+ \left(\sum_{j=0}^{m-1}g_j(x)y^j\right)
\]
where $f_j(x) = \sum_{i}a_{i\/j}x^i$ and
$g_j(x) =  \sum_{i}b_{i\/j}x^i$ are polynomials in
one variable $x$ over ${\comps}$.
Suppose such an element is in the kernel $\ker(\phi)$. By definition this
means that
\begin{eqnarray*}
  {\1}_{\emptyset} & = & \sum_{j=0}^{m-1}
  \left(\sum_{i}a_{i\/j}[[(i+1)\alpha+j\beta,i\alpha+(j+1)\beta]]
  + \sum_{i}b_{i\/j}[\{\alpha+j\beta\}]\right) \\
  & = & \sum_{j=0}^{m-1}
  \left(\sum_{i}a_{i\/j}{\1}_{[(i+1)\alpha+j\beta,i\alpha+(j+1)\beta]}
  + \sum_{i}b_{i\/j}{\1}_{\{i\alpha+j\beta\}}\right).
\end{eqnarray*}
Note that each interval $[(i+1)\alpha+j\beta,i\alpha+(j+1)\beta]$
has length $\beta - \alpha$ and they are all distinct in the first sum,
since $\alpha/\beta = m/n$; a fully reduced fraction and
$i\in\{0,1,\ldots,m-1\}$. 
This means that we can evaluate the above characteristic function on distinct
real numbers, none of the form $i\alpha+j\beta$ and obtain $a_{i\/j} = 0$
for each $i$ and $j$ and hence $f(x,y) = 0$ in ${\comps}[x,y]$.
Consequently $g(x,y) = 0$ in ${\comps}[x,y]$ as well, by the
arguments above and we therefore have that 
\[
f(x,y)z + g(x,y) = \left(\sum_{j=0}^{m-1}f_j(x)y^j\right)z
+ \left(\sum_{j=0}^{m-1}g_j(x)y^j\right) = 0
\]
as a polynomial in ${\comps}[x,y,z]$
This shows that $\ker(\phi) = ((z-x)(z-y), y^m - x^n)$ in this case
and so we have the following
\begin{equation}
\label{eqn:ab-rat}
R_{\alpha,\beta} = M(\{\alpha\},\{\beta\},[\alpha,\beta])
\cong {\comps}[x,y,z]/((z-x)(z-y),y^m - x^n)
\end{equation}
where $\alpha/\beta = m/n$ as an irreducible rational number.

The case when $\alpha < \beta < 0$ and $\alpha/\beta\in {\rats}$
is treated the same way as here above,
since also here there are relatively prime
$m,n\in\nats$ with $\alpha/\beta = m/n$ and hence $y^m - x^n \in\ker(\phi)$.
In exactly the same way one then obtains (\ref{eqn:ab-rat}).

\subsubsection{(D) $\alpha < 0 < \beta$ and $\alpha/\beta\in {\rats}$:}
This case is dealt with in a similar fashion as the previous case (C), but
there are some subtle differences. We will draft this case and
mention the differences.

In addition to $(z-x)(z-y)\in\ker(\phi)$ we have
here that $\alpha/\beta\in{\rats}$ where $\alpha$ and $\beta$ have
opposite sign, and so there are relatively prime
$m,n\in\nats$ with $\alpha/\beta = -m/n$ and hence $x^ny^m - 1 \in\ker(\phi)$
as well. That $\ker(\phi) = ((z-x)(z-y),x^ny^m - 1)$ can be argued in a
similar way as we did in the previous case (C):

Suppose first
that $g(x,y) = \sum_{i,j}b_{i\/j}x^iy^j\in\ker(\phi)$. Again, we can partition
the index set $I_g = \{(i,j) : b_{i\/j}\neq 0\}$ by the distinct values
of $i\alpha + j\beta$, that is $I_g = I_1\cup\cdots\cup I_p$ where each
$I_k = \{ (i,j)\in I_g : i\alpha + j\beta = c_k\}$ where $c_k$ is some
real constant. Hence $g(x,y) = \sum_k g_k(x,y)$ where each
$g_k(x,y) = \sum_{(i,j)\in I_k} b_{i\/j}x^iy^j$. By assumption we have
\[
  {\1}_{\emptyset}
  = \sum_{i,j}b_{i\/j}[[i\alpha + j\beta]]
  = \sum_{i,j}b_{i\/j}{\1}_{[i\alpha + j\beta]}.
\]
By evaluating this characteristic function at $c_k$ for $k = 1,\ldots,p$
we get that $ \sum_{(i,j)\in I_k} b_{i\/j} = 0$ for each $k$ and this condition
is sufficient for $g(x,y)\in\ker(\phi)$. If $(i,j), (i',j')\in I_k$, then
$i\alpha + j\beta = c_k = i'\alpha + j'\beta$ and so
$(i',j') = (i,j) + {\ell}(n,m)$ for some ${\ell}\in{\ints}$. By
induction on $|I_k|$ we obtain, similarly as in (\ref{eqn:m-n}), that 
each $g_k(x,y)$ is a linear combination of binomials of the form
\begin{equation}
\label{eqn:mn-1}  
x^iy^j - x^{i+\ell n}y^{j+\ell m} = x^iy^j(1 - x^{\ell n}y^{\ell m})
\in ((x^ny^m - 1)).
\end{equation}
This implies that $g(x,y) = \sum_k g_k(x,y)\in ((x^ny^m - 1))$.

Again, we note that $\{z^2 - (x+y)z + xy, x^ny^m - 1\}$ forms a
Gr\"{o}bner Basis
for the ideal $((z-x)(z-x),x^ny^m - 1)$ of ${\comps}[x,y,z]$ w.r.t.~LEX
in which $x\preceq y\preceq z$
and so each element in ${\comps}[x,y,z]/\ker(\phi)\cong R_{\alpha,\beta}$
can be expressed in the form $f(x,y)z + g(x,y)$ for some
for some polynomials
$f = \sum_{i,j}a_{i\/j}x^iy^j$ and $g = \sum_{i,j}b_{i\/j}x^iy^j$
in two variables over $\comps$ where neither $f$ nor $g$ have any
terms $x^iy^j$ with $i\geq n$ and $j\geq m$.
Suppose such an element is in the kernel $\ker(\phi)$.
By definition this means that
\[
  {\1}_{\emptyset}
  = \sum_{i,j}a_{i\/j}{\1}_{[(i+1)\alpha + j\beta, i\alpha + (j+1)\beta]}
  + \sum_{i,j}b_{i\/j}{\1}_{\{i\alpha+ j\beta\}}.
\]
As in the previous case we note that each interval
$[(i+1)\alpha+j\beta,i\alpha+(j+1)\beta]$
has length $\beta - \alpha$ and they are all distinct in the first sum,
since $\alpha/\beta = -m/n$ where $m/n$ is a fully reduced
fraction and none of the nonzero coefficients $a_{i\/j}$ has both
$i\geq n$ and $j\geq m$. 
This means that we can evaluate the above characteristic function on distinct
real numbers, non of the form $i\alpha+j\beta$ and obtain $a_{i\/j} = 0$
for each $i$ and $j$ and hence $f(x,y) = 0$ in ${\comps}[x,y]$.
Consequently $g(x,y) = 0$ in ${\comps}[x,y]$ as well, by the
arguments above and we therefore have that 
$f(x,y)z + g(x,y) = 0$
as a polynomial in ${\comps}[x,y,z]$
This shows that $\ker(\phi) = ((z-x)(z-y), x^ny^m - 1)$ in this case
and so we have, similarly to (\ref{eqn:ab-rat}), that
\[
R_{\alpha,\beta} = M(\{\alpha\},\{\beta\},[\alpha,\beta])
\cong {\comps}[x,y,z]/((z-x)(z-y),x^ny^m - 1)
\]
where $\alpha/\beta = -m/n$ as an irreducible rational number.

We summarize these findings from cases (A), (B), (C) and (D)
in the following.
\begin{proposition}
\label{prp:4-rings}
For $\alpha,\beta\in\reals$ we have the following isomorphism classes
for the Minkowski ring
$R_{\alpha,\beta} := M(\{\alpha\},\{\beta\},[\alpha,\beta])$
generated by a closed interval and its endpoints.
\begin{enumerate}
\item If $0 = \alpha < \beta$ or $\alpha < \beta = 0$, then
\[
R_{\alpha,\beta} \cong {\comps}[x,y]/((y-1)(y-x).
\]
\item If $0\not\in\{\alpha,\beta\}$ and $\alpha/\beta\not\in {\rats}$, then 
\[
R_{\alpha,\beta} \cong {\comps}[x,y,z]/((z-x)(z-y).
\]
\item If $\alpha/\beta = m/n \in {\rats}$ where $m$ and $n$ are relatively
prime natural numbers, then
\[
R_{\alpha,\beta} \cong {\comps}[x,y,z]/((z-x)(z-y),y^m - x^n).
\]
\item If $\alpha/\beta = -m/n \in {\rats}$ where $m$ and $n$ are relatively
prime natural numbers, then
\[
R_{\alpha,\beta} \cong {\comps}[x,y,z]/((z-x)(z-y),x^ny^m - 1).
\]
\end{enumerate}
\end{proposition}
\begin{rmks*}
  (i) Note that each of the four presented as the kernels $\ker(\phi)$
  in the above Proposition~\ref{prp:4-rings} is a power-closed ideal
  although only the second one of them is homogeneous. 
  (ii) In the same way as
  $R_{\alpha,\beta} = M(\{\alpha\},\{\beta\},[\alpha,\beta])
  \cong {\comps}[x,y,z]/\ker(\phi)$, where
$\phi : {\comps}[x,y,z] \twoheadrightarrow R_{\alpha,\beta}$ is by
$\phi(x) = {\1}_{\{\alpha\}}$, $\phi(y) = {\1}_{\{\beta\}}$ and
$\phi(z) = {\1}_{[\alpha,\beta]}$, then we can let
$R_{\alpha,\beta}^{\pm} := {\comps}[x,y,z]^{\pm}/\ker(\phi)$, where
$\phi : {\comps}[x,y,z]^{\pm} \twoheadrightarrow R_{\alpha,\beta}^{\pm}$
is defined for the Laurent polynomials
in the same way as for $R_{\alpha,\beta}$. In this way, the
above Proposition~\ref{prp:4-rings} also holds
when the polynomial rings ${\comps}[x,y]$ and ${\comps}[x,y,z]$
are replaced by the corresponding Laurent polynomial rings
${\comps}^{\pm}[x,y]$ and ${\comps}[x,y,z]^{\pm}$ respectively, since
in every one of the four cases $\ker(\phi)$ in Proposition~\ref{prp:4-rings}
does not contain a monomial generator.
\end{rmks*}
With the setup from the above remarks we 
have directly from Proposition~\ref{prp:4-rings} the following slightly
simpler presentation since the last two cases become identical if one allows
negative powers of $x, y$ and $z$. We state this as the following
corollary.
\begin{corollary}
\label{cor:3-rings}
For $\alpha,\beta\in\reals$ we have the following isomorphism classes
for the Laurent Minkowski ring $R_{\alpha,\beta}^{\pm}$
generated by a closed interval and its endpoints.
\begin{enumerate}
\item If $0 = \alpha < \beta$ or $\alpha < \beta = 0$, then
\[
R_{\alpha,\beta}^{\pm} \cong {\comps}[x,y]^{\pm}/((y-1)(y-x).
\]
\item If $0\not\in\{\alpha,\beta\}$ and $\alpha/\beta\not\in {\rats}$, then 
\[
R_{\alpha,\beta}^{\pm} \cong {\comps}[x,y,z]^{\pm}/((z-x)(z-y).
\]
\item If $\alpha/\beta = m/n \in {\rats}$ where $m$ and $n$ are relatively
prime integers, then
\[
R_{\alpha,\beta}^{\pm} \cong {\comps}[x,y,z]^{\pm}/((z-x)(z-y),y^m - x^n).
\]
\end{enumerate}
\end{corollary}
Apriori, it is tempting to believe that in general the Minkowski ring of
of a finite collection of closed sets in $W = {\reals}^d$ is
isomorphic to the Minkowski ring of an affine bijective image of these
closed sets collectively. In most cases this is indeed the case.
However, we will now show that the four types of $\comps$-algebras from
Proposition~\ref{prp:4-rings} are different and always
non-isomorphic as rings. This means that Minkowski rings depend very
much on the geometry, or the linear dependence of the closed subsets of
${\reals}$ over either ${\nats}$ or ${\ints}$,
of the polytopes in question as well as their
combinatorics. The rest of this section is devoted to demonstrating
that the rings from Proposition~\ref{prp:4-rings} are non-isomorphic
as rings and not merely as ${\comps}$-algebras. 

For convenience let
$R_1 := {\comps}[x,y]/((y-1)(y-x))$,
$R_2 := {\comps}[x,y,z]/((z-x)(z-y))$, 
$R_3(m,n) := {\comps}[x,y,z]/((z-x)(z-y),y^m - x^n)$ and
$R_4(m,n) := {\comps}[x,y,z]/((z-x)(z-y),x^ny^m - 1)$,
where $m, n\in\nats$ are relatively prime,
be the four ${\comps}$-algebras presented
in Proposition~\ref{prp:4-rings}. We will show that they are
mutually non-isomorphic rings. In particular we will show
that $R_3(m,n)$ and $R_4(m',n')$ are always non-isomorphic for
any pairs $(m,n)$ and $(m',n')$ of relatively prime natural numbers.

For $R_1$ note that the bijective
change of variables $x' = y - x$ and $y' = y - 1$ shows that
$R_1 \cong {\comps}[x',y']/(x'y')$ as $\comps$ -algebras.
Likewise, for $R_2$ the bijective change of variables $x' = x-z$, $y' = y-z$
and $z' = z$ shows that
$R_2 \cong {\comps}[x',y',z']/(x'y')$ as $\comps$-algebras
and hence $R_2 = R_1[z']$, the polynomial algebra in one variable
$z'$ over $R_1$. Taking the Krull dimension of both the rings
$R_1$ and $R_2$ it is
well-known that $\dim(R_1[z']) = \dim(R_1) + 1$, since $R_1$ is a Noetherian
ring~\cite[p.~126]{Atiyah-Macdonald}, and hence $R_1$ and $R_2$ are not
isomorphic as rings nor as $\comps$-algebras. In fact, looking at the
original presentation of $R_1$ and $(y-1)(y-x) = y^2 - (1+x)y +x$ as the
sole element of a Gr\"{o}bner Basis for the ideal $((y-1)(y-x))$, w.r.t.~the
term order LEX in which $x\preceq y$, we then obtain
\[
\dim(R_1)
= \dim({\comps}[x,y]/(y^2))
= \dim({\comps}[x,y]/(y))
= \dim({\comps}[x])
= 1,
\]
and hence $\dim(R_2) = 2$. Similarly for $R_3(m,n)$ we see that
$\{z^2 - (x+y)z + xy, y^m - x^n\}$ forms a Gr\"{o}bner Basis
for the ideal $((z-x)(z-x),y^m - x^n)$ of ${\comps}[x,y,z]$
w.r.t.~LEX in which $x\preceq y\preceq z$ and hence 
\[
\dim(R_3(m,n))
= \dim({\comps}[x,y,z]/(z^2,y^m))
= \dim({\comps}[x,y,z]/(z,y))
= \dim({\comps}[x])
= 1,
\]
and hence $R_2\not\cong R_3(m,n)$. We now show that $R_1\not\cong R_3(m,n)$
despite they both have Krull dimension of one. For that we look closer
at the lattice of prime ideals of both $R_1$ and $R_3(m,n)$. 

By the bijective change of variables from above we have
$R_1 \cong {\comps}[x,y]/(xy)$ as $\comps$-algebras. If $P$ is a
minimal prime ideal of ${\comps}[x,y]$ that contains the element $xy$
then either $x\in P$ or $y\in P$. If $x\in P$ then $(x)\subseteq P$
and hence $P = (x)$ by the minimality of $P$. Similarly if $y\in P$
then $P = (y)$. Consequently $R_1$ has exactly two minimal
prime ideals $\overline{(x)}$ and $\overline{(y)}$ and their
sum is $\overline{(x,y)}$ a maximal ideal of $R_1$.
\begin{claim}
\label{clm:sum-max}
The sum of the minimal prime ideals of $R_1$ is a maximal ideal of $R_1$.
\end{claim}

Regarding the ring $R_3(m,n)$ we fist note that by bijective change
of variables,
similar to what we considered here above for $R_2$, we obtain that
$R_3(m,n) \cong {\comps}[x,y,z]/(xy, (z+y)^m - (z+x)^n)$. As for $R_2$ consider
a minimal prime ideal $P$ of ${\comps}[x,y,z]$ that contains the ideal
$I = (xy, (z+y)^m - (z+x)^n)$. Since $xy\in P$ we have either $x\in P$ or
$y\in P$.

If $x\in P$ then we have
$(x, (z+y)^m - (z+x)^n) = (x, (z+y)^m - z^n)\subseteq P$. Note that
\begin{equation}
  \label{eqn:int-domain}
{\comps}[x,y,z]/(x, (z+y)^m - z^n) = {\comps}[y,z]/((z+y)^m - z^n)
\cong {\comps}[y,z]/(y^m - z^n).
\end{equation}
The following is from~\cite[Lemma 6.2]{A-Lawrence}.
\begin{lemma}
\label{lmm:6.2}
A binomial
$b = x_1^{p_1}\cdots x_h^{p_h} - y_1^{q_1}\cdots y_k^{q_k}
\in{\comps}[x_1,\ldots,x_h,y_1,\ldots,y_k]$ is irreducible in both
${\comps}[x_1,\ldots,x_h,y_1,\ldots,y_k]$ and in
${\comps}[x_1,\ldots,x_h,y_1,\ldots,y_k]^{\pm}$ if and only if
$\gcd(p_1,\ldots,p_h,q_1,\ldots,q_k) = 1$.
\end{lemma}
By Lemma~\ref{lmm:6.2} we have that the polynomial $y^m - z^n$
is irreducible both in ${\comps}[y,z]$ and ${\comps}[y,z]^{\pm}$
Since these are both UFDs then $y^m - z^n$ is a prime element in
both rings and hence $(y^m - z^n)$ is a prime ideal in both rings.
In particular ${\comps}[y,z]/(y^m - z^n)$ is an integral domain.
By (\ref{eqn:int-domain}) the ideal $(x, (z+y)^m - z^n)$ is therefore
a prime ideal of ${\comps}[x,y,z]$. By the minimality of $P$ we therefore
have that $P = (x, (z+y)^m - z^n)$ in this this case.

If $y\in P$ we get similarly that $P = (y, z^m - (z+x)^n)$.
Consequently there are exactly two minimal
prime ideals of ${\comps}[x,y,z]$ that contain
the ideal $(xy, (z+y)^m - (z+x)^n)$ and their sum is
\[
(x, (z+y)^m - z^n) + (y, z^m - (z+x)^n) = (x,y,z^m - z^n)
\]
which is not a maximal ideal of ${\comps}[x,y,z]$ since
${\comps}[x,y,z]/(x,y,z^m-z^n) \cong {\comps}[z]/(z^m-z^n)$
which is never a field when $m$ and $n$ are relatively prime natural numbers.
It is not even an integral domain unless $m = n = 1$.
The ideal $(x,y,z^m - z^n)$ of ${\comps}[x,y,z]$ corresponds to the
ideal $\Sigma_3$ of $R_3(m,n)$; the sum of the two minimal primes of $R_3(m,n)$
and, assuming $2\leq m < n$ we have 
\begin{equation}
  \label{eqn:R3-S3}
R_3(m,n)/\Sigma_3 \cong {\comps}[z]/(z^m-z^n) = {\comps}[z]/(z^m(1-z^{n-m}))
\cong {\comps}[z]/(z^n)\times {\comps}^{m-n}
\end{equation}
as ${\comps}$-algebras and so $R_3(m,n)/\Sigma_3$ contains a nilpotent element;
$(\overline{z},0,\ldots,0)$ from the right hand side of (\ref{eqn:R3-S3}).
As a result we have the following.
\begin{claim}
\label{clm:sum-not-max}
The sum of the minimal prime ideals $\Sigma_3$ of $R_3(m,n)$ is a not a
maximal ideal of $R_3(m,n)$. Further, if $m,n\geq 2$, then $R_3(m,n)/\Sigma_3$
contains a nilpotent element.
\end{claim}
By Claims~\ref{clm:sum-max} and~\ref{clm:sum-not-max} we see that the
${\comps}$-algebras $R_1, R_2$ and $R_3(m,n)$ are always mutually non-isomorphic
as rings.

For $R_4(m,n)$ we see that
$\{z^2 - (x+y)z + xy, x^ny^m - 1\}$ forms a Gr\"{o}bner Basis
for the ideal $((z-x)(z-x),x^ny^m - 1)$ of ${\comps}[x,y,z]$
w.r.t.~LEX in which $x\preceq y\preceq z$ and so 
\begin{eqnarray*}
  \dim(R_4(m,n))
  & = & \dim({\comps}[x,y,z]/(z^2,x^ny^m)) \\
  & = & \dim({\comps}[x,y,z]/(z,xy)) \\
  & = & \dim({\comps}[x,y]/(xy)) \\
  & = & \dim(R_1) \\
  & = & 1,
\end{eqnarray*}
and hence $R_4(m,n)\not\cong R_2$. We need to show that
$R_4(m,n)\not\cong R_1,R_3(m',n')$ for any $m,m',n,n'\in\nats$. 

As with the rings $R_2$ and $R_3(m,n)$, we have that with a bijective
change of variables that the ring
$R_4(m,n) \cong {\comps}[x,y,z]/(xy, (z+x)^n(z+y)^m - 1)$ and that there are
exactly two minimal prime ideals in ${\comps}[x,y,z]$ that contain
the ideal $J = (xy, (x+z)^n(y+z)^m - 1)$, namely
$(x, z^n(z+y)^m - 1)$ containing $x$ and
$(y, (z+x)^nz^m - 1)$ containing $y$. The minimal ideal containing
both of these ideals, their sum is then
\[
(x, z^n(z+y)^m - 1) + (y, (z+x)^nz^m - 1) = (x,y,z^{m+n} - 1)
\]
which is not a maximal ideal of ${\comps}[x,y,z]$ since
${\comps}[x,y,z]/(x,y,z^{m+n}-1) \cong {\comps}[z]/(z^{m+n}-1)$.
Again, the ideal $(x,y,z^{m+n} - 1)$ of ${\comps}[x,y,z]$ corresponds to the
ideal $\Sigma_4$ of $R_4(m,n)$; the sum of the two minimal primes of $R_4(m,n)$
and, since $z^{m+n} - 1$ is a separable polynomial in ${\comps}[z]$, then
we have 
\begin{equation}
  \label{eqn:R4-S4}
R_4(m,n)/\Sigma_4 \cong {\comps}[z]/(z^{m+n}-1) \cong {\comps}^{m+n}
\end{equation}
as ${\comps}$-algebras,
the Cartesian product of $m+n$ copies of ${\comps}$ and so
$R_4(m,n)/\Sigma_4$ never contains a nilpotent element.
By Claim~\ref{clm:sum-max} and (\ref{eqn:R4-S4}) we have
$R_4(m,n)\not\cong R_1$. Also,
if $m,n\geq 2$ we see by (\ref{eqn:R3-S3}) and (\ref{eqn:R4-S4})
that $R_3(m,n)$ is never isomorphic to the ring
$R_4(m',n')$ for any $m',n'\in\nats$.

Finally, consider $R_3(m,n)$ when either $m=1$ or $n=1$,
say $m=1$ and $n\geq 1$. In this case we have by definition that
$R_3(1,n) = {\comps}[x,y,z]/((z-x)(z-y),y - x^n)
= {\comps}[x,z]/((z-x)(z-x^n))$. By completing the square of
$(z-x)(z-x^n) = (z - (x+x^n)/2)^2 - ((x-x^n)/2)^2$ and
using a change of variable $y := z - (x+x^n)/2$ we get that
$R_3 (1,n) \cong {\comps}[x,y]/(y^2 - f(x)^2)$ where
$f(x) = (x-x^n)/2$. Note that $n\geq 2$ since
$\alpha\neq\beta$ in $R_{\alpha,\beta}$ from Proposition~\ref{prp:4-rings}
and hence $f(x)$ is a polynomial in $x$ of degree at least $2$. 
\begin{lemma}
\label{lmm:units}
For any $f\in\comps[x]$ of positive degree, the group of units
is $U({\comps}[x,y]/(y^2 - f(x)^2)) = {\comps}^*$,\;
the multiplicative group of the non-zero elements of $\comps$.
\end{lemma}
\begin{proof} (Sketch:)
Note that each element in ${\comps}[x,y]/(y^2 - f(x)^2)$ has a unique
representation as $a(x) + yb(x)$ for some $a(x),b(x)\in {\comps}[x]$
and hence there is a well defined field-norm-like map
$N : {\comps}[x,y]/(y^2 - f(x)) \rightarrow {\comps}[x]$ given by
$N(a(x) + yb(x)) = a(x)^2 - f(x)^2b(x)^2$ that respects multiplication.

Suppose $u = a(x) + yb(x)$ is a unit, then there is a $v$ with $uv = 1$
and hence $1 = N(1) = N(uv) = N(u)N(v)$ as polynomials in ${\comps}[x]$
and so $N(u),N(v)\in {\comps}^*$. In particular 
\[
(a(x) - f(x)b(x))(a(x) + f(x)b(x)) = N(u)\in {\comps}^*.
\]
Since $a(x) - f(x)b(x), a(x) + f(x)b(x)\in {\comps}[x]$ we must have
both $a(x) - f(x)b(x), a(x) + f(x)b(x)\in {\comps}^*$ and therefore 
$a(x), f(x)b(x)\in {\comps}$. Since $f(x)$ has degree of one or more
we must have $b(x) = 0$, the zero polynomial and hence $u\in {\comps}^*$.
\end{proof}  
Suppose now that $R_4(m'n')\cong R_3(1,n)$ as rings where $m'$ and $n'$ are
relatively prime natural numbers and $n\geq 2$.
This means that there is a ring isomorphism
$\psi : {\comps}[x,y,z]/((z-x)(z-y), x^{n'}y^{m'} - 1)
\rightarrow {\comps}[x,y]/(y^2 - f(x)^2)$ where $f(x) = (x-x^n)/2$. Since
both $x$ and $y$ are units in $R_4(m',n')$ then their images under $\psi$
are also units in $R_3(1,n)$. Hence by Lemma~\ref{lmm:units}
we must have $\psi(x),\psi(y)\in{\comps}^*$ and so
\[
0 = \psi(0) = \psi((z-x)(z-y))
= \psi(z)^2 -(\psi(x) + \psi(y))\psi(z) + \psi(x)\psi(y)
\]
in the ring $R_3(1,n) = {\comps}[x,y]/(y^2 - f(x)^2)$. This means that
the element $\psi(z)$ is a root of a quadratic polynomial over ${\comps}$.
Since every element of $R_3(1,n) = {\comps}[x,y]/(y^2 - f(x)^2)$ has a
representation as $p(x) + q(x)y$ where $p$ and $q$ are uniquely determined
polynomials in $x$ over ${\comps}$, then $\psi(z) = p(x) + q(x)y$ for
some $p,q\in {\comps}[x]$. It is easy to see that $p(x) + q(x)y$ is a root
of a quadratic polynomial over ${\comps}$ in $R_3(1,n)$ if and only
if $q(x) = 0$ as a polynomial, in which case $\psi$ is not surjective and
therefore not an isomorphism.

We summarize in the following.
\begin{theorem}
\label{thm:non-isom}
The four types of rings $R_1, R_2, R_3(m,n)$ and $R_4(m',n')$
from Proposition~\ref{prp:4-rings} are never isomorphic when $(m,n)$ and
$(m',n')$ are pair of relatively prime natural numbers.
\end{theorem}
As a result, although it doesn't present a huge difficulty,
one should be somewhat careful when talking about ``the'' Minkowski
ring of a given finite collection of closed sets of $W$, in particular
the collection consisting of a convex polytope and all its faces.

We will often for simplicity assume that exactly one
endpoint of a closed interval in $W = {\reals}$ is equal to zero, precisely
to obtain the first of the four types of Minkowski ring
in Proposition~\ref{prp:4-rings} with fewer variables in its presentation.

\subsection{The positive integer box case}
Here we consider the set ${\cal{B}^+_d}$ of all $d$-dimensional boxes
$B = \prod_{i=1}^d [a_i,b_i]$ in $W = {\reals}^d$ where $a_i$ and $b_i$
are nonnegative integers. Clearly ${\cal{B}}_d^+$ is
closed under taking Minkowski sums in ${\reals}^d$. By allowing
degenerate boxes, where $a_i=b_i$ is allowed, we see that all the
faces of the $d$-box $B$ as a polytope are also boxes. Hence,
${\cal{B}}^+_d$ contains all ${\ell}$-dimensional boxes 
$B = \prod_{i=1}^d [a_i,b_i]$ in $W$ where ${\ell}\leq d$ and all their
faces as well. Also note that
$M({\cal{B}}_1^+) = M(\{1\},[0,1])$ since the characteristic function
of each nonnegative point ${\1}_{\{a\}}$ and each interval ${\1}_{[a,b]}$
is contained in $M(\{1\},[0,1])$. Note that
any bounded closed $d$-box can be written as a direct sum 
\[
B = [a_1,b_1]\times\cdots\times[a_d,b_d] = \bigoplus_{i=1}^d \pi_i(B),
\]
where $\pi_i : {\reals}^d \rightarrow {\reals}^d$
denotes the projection onto the $i$-th coordinate
$\tilde{x}\mapsto x_i\tilde{e}_i$. Hence,
$B$ is a direct Minkowski sum its edges. By the directness of
the Minkowski sum in this case, $F$ is a face of $B$ if and only
if $F =  \prod_{i=1}^d F_i$ where each $F_i\in \{[\{a_i\}, \{b_i\}, [a_i,b_i]\}$
is a face of the 1-polytope $[a_i,b_i]$. If $x_i$ represents the characteristic
function ${\1}_{\{\tilde{e}_i\}}$ and $y_i$ represents
${\1}_{\pi_i([0,1]^d)} = [\pi_i([0,1]^d)]$, then each $d$-box $B$ in
${\cal{B}}_d^+$ can be represented by the monomial
$x_1^{a_1}\cdots x_d^{a_d}y_1^{b_1-a_1}\cdots y_d^{b_d-a_d}$. Note that
this includes the degenerate case when $a_i=b_i$ and so each $d$-box
and each of its proper faces has the form of a monomial in $x_i$
and $y_i$. Since the Minkowski ring for the $i$-th coordinate edge
interval is by Proposition~\ref{prp:4-rings} given by
$M(\{1\},[0,1]) \cong {\comps}[x_i,y_i]/((y_i-1)(y_i-x_i))$ we can
use the reduction $y_i^2 = (x_i + 1)y_i - x_i$, as in Subsection~\ref{ssec:0b},
and we see that each element in 
${\comps}[x_1,\ldots,x_d,y_1,\ldots,y_d]/((y_1-1)(y_1-x_1),\ldots,
(y_d-1)(y_d-x_d))$ has a reduced form of
\begin{equation}
  \label{eqn:d-reduced}
f(\tilde{x},\tilde{y})
= \sum_{A\subseteq \{1,\ldots,d\}}f_A(\tilde{x})\prod_{i\in A}y_i,
\end{equation}
where $\tilde{x} = (x_1,\ldots,x_d)$, $\tilde{y} = (y_1,\ldots,y_d)$
and $f_A$ is a polynomial in ${\comps}[x_1,\ldots,x_d]$.
Without rehashing previous arguments, an expression of this above
form represents the characteristic function of translated $q$-boxes
in ${\cal{B}}_d^+$ where $q\in \{0,\ldots,d\}$, where the relative
interiors of the cubes of the same dimension are disjoint for each $q$.
Hence, if $f(\tilde{x},\tilde{y}) = 0$ in the above quotient ring, then we
must have $f = 0$ in ${\comps}[x_1,\ldots,x_d,y_1,\ldots,y_d]$.
As a result, the reduced form from (\ref{eqn:d-reduced}) is a normal form
and there are no other relations that elements $x_i$ and $y_i$ satisfy.
Hence, for the Minkowski ring $M({\cal{B}}_d^+)$ of
all the $d$-boxes and its faces in $W$ we therefore have
the following.
\begin{corollary}
\label{cor:d-box}
The surjection
$\phi : {\comps}[x_1,\ldots,x_d,y_1,\ldots,y_d]\twoheadrightarrow
M({\cal{B}}_d^+)$
given by $\phi(x_i) = {\1}_{\{\tilde{e}_i\}}$ and
$\phi(y_i) = [\pi_i([0,1]^d)]$ has
$\ker(\phi) = ((y_1-1)(y_1-x_1),\ldots,(y_d-1)(y_d-x_d))$ and
so
\[
M({\cal{B}}_d^+) \cong
{\comps}[x_1,\ldots,x_d,y_1,\ldots,y_d]/((y_1-1)(y_1-x_1),\ldots,
(y_d-1)(y_d-x_d)).
\]
In particular we have
\[
M({\cal{B}}_d^+) = \bigotimes_{i=1}^d M({\cal{B}}_1^+)
= \bigotimes_{i=1}^d M(\{1\},[0,1]),
\]
where the tensor product is taken over $\comps$.
\end{corollary}
The above corollary is a special case of what we will discuss in more
detail in Section~\ref{sec:Cartesian}.

\subsection{The general integer box case}
Similarly, if ${\cal{B}}_d$ denotes the set of all $d$-dimensional boxes
$B = \prod_{i=1}^d [a_i,b_i]$ where $a_i,b_i\in {\ints}$ are integers, then
$M({\cal{B}}_1) = M(\{-1\}, \{1\}, [0,1])$ and
\[
M({\cal{B}}_1) \cong {\comps}[x,x^{-1},y]/((y-1)(y-x))
\cong {\comps}[x,y]^{\pm}/((y-1)(y-x))
\]
where the isomorphism is given as for $M(\{1\}, [0,1])$, namely
$x\mapsto {\1}_{\{1\}}$ and $y\mapsto {\1}_{[0,1]}$. Note that in
${\comps}[x,y]^{\pm}/((y-1)(y-x))$ we have $y^{-1} = 1 + x^{-1} - x^{-1}y$
and so the inverse of $y$ is solely in terms of $1, x^{-1}$ and $y$.
This corresponds to the inverse of ${\1}_{[0,1]}$ being
$-{\1}_{]-1,0[}$
in $M({\cal{B}}_1)$, which is a special case of the inverse of
a general polytope $P$ of dimension $d$
in a Minkowski ring containing $[P]$ and its
inverse is given by $[P]^{-1} = (-1)^d[-{\mathring{P}}]$, where $\mathring{P}$
denotes the {\em relative interior} of $P$ in
${\reals}^d$~\cite[p.~13]{Lawrence}, \cite[p.~3]{Morelli}.
Since we similarly have
\[
M({\cal{B}}_d) = \bigotimes_{i=1}^d M({\cal{B}}_1)
= \bigotimes_{i=1}^d M(\{-1\},\{1\},[0,1]),
\]
we then have analogously the following.
\begin{corollary}
\label{cor:d-box-pm}
The surjection
$\phi : {\comps}[x_1,\ldots,x_d,y_1,\ldots,y_d]^{\pm}\twoheadrightarrow
M({\cal{B}}_d)$
given by $\phi(x_i) = {\1}_{\{\tilde{e}_i\}}$ and
$\phi(y_i) = [\pi_i([0,1]^d)]$ has
$\ker(\phi) = ((y_1-1)(y_1-x_1),\ldots,(y_d-1)(y_d-x_d))$ and
so
\[
M({\cal{B}}_d) \cong
{\comps}[x_1,\ldots,x_d,y_1,\ldots,y_d]^{\pm}/((y_1-1)(y_1-x_1),\ldots,
(y_d-1)(y_d-x_d)).
\]
\end{corollary}
Note that by the above Corollary~\ref{cor:d-box-pm} we clearly have that
$M^{\pm}({\cal{B}}_d^+) = M({\cal{B}}_d)$ for each $d\geq 1$, that is,
each $x_i$ and $y_i$ has a multiplicative inverse in $M({\cal{B}}_d)$.

As $d$-dimensional boxes provide one of the simplest examples of a class
of closed sets that are closed under taking Minkowski sums, what
Corollaries~\ref{cor:d-box} and~\ref{cor:d-box-pm} state is basically
the following:
\begin{quote}
  All the relations in the Minkowski ring of closed boxes
  in ${\ints}^d$ are derivations from the relations that the
  $1$-dimensional sides of the boxes satisfy as closed intervals
  with integer endpoints.
\end{quote}
\begin{rmk*}
Note that for $d=1$ we have a fully reduced (or normal) form $f(x)y + g(x)$ for
an element in $M({\cal{B}}_1)$, as indicated in Subsection~\ref{ssec:0b}
of the proof of Proposition~\ref{prp:4-rings}. From a more geometric point
of view however, there is another unique normal form for
of each closed interval in ${\cal{B}_1}$
by noting that each closed interval in ${\cal{B}_1}$ is a disjoint union
of elements from ${\ints}$ and translates of $]0,1[$, the relative interior of
$[0,1]$. In other words, {\em each closed interval in ${\cal{B}}$
can be tiled by translates of $\{1\}$ and $]0,1[$}.
Note that the indicator function of $]0,1[$ is represented by
$\mathring{y} = y - 1 - x$ in ${\comps}[x,y]^{\pm}/((y-1)(y-x))$ so
more specifically, a tiling of $[0,n]$ corresponding to $y^n$ is given by
\begin{equation}
\label{eqn:y-tiling}
y^n = \sum_{i=0}^n x^i + \mathring{y}\sum_{i=0}^{n-1}x^i.
\end{equation}
Since this form for $y^n$ corresponds to a geometric tiling, the elements $x^i$
and $\mathring{y}x^i$ correspond to indicator functions of disjoint
subsets of ${\reals}$. 
That (\ref{eqn:y-tiling}) follows from the defining relation
$(y-1)(y-x) = 0$ can be
best seen by first localizing at $1-x$ and providing it with an
inverse\footnote{this can be argued since $S = \{{(1-x)}^n : n\in\ints\}$ is
  multiplicatively closed and hence adjoining $S^{-1}$ is legitimate}
and then rewriting (\ref{eqn:y-tiling}) as
$(1-x)y^n = 1-x^{n+1} + (y - 1 - x)(1-x^n)$ which indeed follows from
$(y-1)(y-x) = 0$.
\end{rmk*}
In fact, it also follows directly from the defining relation $(y-1)(y-x) = 0$
that $\mathring{y} = -xy^{-1}$ 
we can tile each closed interval by indicator functions corresponding to
translates of $1$ and $y^{-1}$ as well. We summarize in the
following observation.
\begin{observation}
\label{obs:dim-tiling}
For $M({\cal{B}}_d)$ presented as 
\[
M({\cal{B}}_d) \cong
{\comps}[x_1,\ldots,x_d,y_1,\ldots,y_d]^{\pm}/((y_1-1)(y_1-x_1),\ldots,
(y_d-1)(y_d-x_d)).
\]
as in Corollary~\ref{cor:d-box-pm}, we have that any closed
interval, as a translate of $y_i^n$ for some $i$ and $n$, can
be tiled as in (\ref{eqn:y-tiling}) by translates of $1$ and
$\mathring{y_i}$ on one hand and also by translates of
$1$ and $y_i^{-1}$ on the other hand. Further, any
tiling of such translates of intervals $y_i^n$ will yield a
corresponding $d$-dimensional tiling of any $n_1\times\cdots\times n_d$
closed box in $M({\cal{B}}_d)$ which is a product of closed
intervals of lengths $n_1,\ldots, n_d$.
\end{observation}
By (\ref{eqn:y-tiling}) we see that any Laurent polynomial
$f(x,y)\in {\comps}[x,x^{-1}y]$ mapped to $M({\cal{B}}_1)$ can be written
as a unique Laurent polynomial
$\mathring{f}(x,\mathring{y})\in {\comps}[x,x^{-1},\mathring{y}]$
or a unique Laurent polynomial $f_{-1}(x,y^{-1})\in{\comps}[x,x^{-1},y^{-1}]$.
Since monomials $x^i\mathring{y}$ where $i\in\ints$ on one hand and
$x^iy^{-1}$ where $i\in\ints$ on the other hand correspond to disjoint
open intervals, we see that $f\in ((y-1)(y-x))$, the kernel of the
surjection ${\comps}[x,y]^{\pm}\twoheadrightarrow M^{\pm}(\{1\},[0,1])$,
if and only both $\mathring{f}$ and $f_{-1}$ are the zero Laurent polynomials.

Such geometric tiling considerations will be exploited in the next section.

\section{The affine  arrangement and identities}
\label{sec:Coxeter}

The goal of this section is firstly to consider explicitly
a more interesting example of
a Minkowski ring, namely the Minkowski ring of certain convex sets in
the Euclidean plane ${\reals}^2$ and secondly to investigate
properties of some identities that hold in general Minkowski rings.

\subsection{The Coxeter arrangement}
Here we consider the Minkowski ring of convex sets in the
Euclidean plane ${\reals}^2$
that are formed by the intersection of horizontal, vertical and slanted
half-planes of the form
\begin{equation}
\label{eqn:half-planes}
\begin{split}
H_h^-(\ell) = \{ (x,y)\in {\reals}^2 : y\leq \ell \}, \\
H_h^+(\ell) = \{ (x,y)\in {\reals}^2 : y\geq \ell \}, \\
H_v^-(\ell) = \{ (x,y)\in {\reals}^2 : x\leq \ell \}, \\
H_v^+(\ell) = \{ (x,y)\in {\reals}^2 : x\geq \ell \}, \\
H_s^-(\ell) = \{ (x,y)\in {\reals}^2 : y\leq x + \ell \}, \\
H_s^+(\ell) = \{ (x,y)\in {\reals}^2 : y\geq x + \ell \},
\end{split}
\end{equation}
where $\ell\in\ints$ is an arbitrary integer. The convex sets
are here bounded by a discrete collection of lines in ${\reals}^2$
of the form $y=\ell$, $x=\ell$ and $y = x+\ell$, where $\ell\in\ints$.
Such a discrete collection of lines (or hyperplanes in ${\reals}^d$)
is called an {\em affine Coxeter arrangement}. These convex sets formed
in this way are exactly those that are bounded by the closed convex
polygons in ${\reals}^2$ such that each polygon has (i) its vertices
in ${\ints}^2$ and (ii) each side has slope in $\{0, 1, \infty\}$,
that is, each side is horizontal, vertical or of slope one.
For a more symmetric visualization we apply a linear
transformation to obtain an equivalent collection
of convex polygons with vertices in
${\ints}(1,0) + {\ints}(\frac{1}{2},\frac{\sqrt{3}}{2})$,
the regular triangular grid in ${\reals}^2$ and sides of slopes
$\{0, \pm\sqrt{3}\}$. We will denote this collection of closed
convex sets by ${\cal{A}}$ and the corresponding
Minkowski ring by $M^{\pm}({\cal{A}})$.
Let $\llbracket\tilde{p}_1,\ldots,\tilde{p}_n\rrbracket$
denote the closed convex hull of points $\tilde{p}_1,\ldots,\tilde{p}_n$
and so the corresponding characteristic
function is ${\1}_{\llbracket\tilde{p}_1,\ldots,\tilde{p}_n\rrbracket}$. If
$O = (0,0)$, ${A} = (1,0)$
and ${B} = (\frac{1}{2},\frac{\sqrt{3}}{2})$ then,
since $d=2$, our convex sets are all
tiled by open equilateral triangles with side length of $1$, side slopes
of $0$ or $\pm\sqrt{3}$ and vertices from
${\ints}(1,0) + {\ints}(\frac{1}{2},\frac{\sqrt{3}}{2})$.
Therefore we have
\[
M^{\pm}({\cal{A}}) = M^{\pm}(\llbracket{A}\rrbracket, \llbracket{B}\rrbracket, \llbracket{O}, {A}\rrbracket,
\llbracket{O},{B}\rrbracket, \llbracket{A},{B}\rrbracket,
\llbracket{O},{A},{B}\rrbracket)
\]
\comment{ 
\[
\begin{split}
  M({\cal{A}}) = & M\left(\{(1,0)\},
  \left\{\left(\frac{1}{2},\frac{\sqrt{3}}{2}\right)\right\},
              [(0,0), (1,0)],
              \left[(0,0), \left(\frac{1}{2},\frac{\sqrt{3}}{2}\right)\right],
              \left[(1,0), \left(\frac{1}{2},\frac{\sqrt{3}}{2}\right)\right],
              \right. \\
              & \left.\left[(0,0), (1,0),
                \left(\frac{1}{2},\frac{\sqrt{3}}{2}\right)\right]\right)
\end{split}
\]
}
and a $\comps$-algebra epimorphism
$\phi : {\comps}[x_1,x_2,y_1,y_2,y_3,z]^{\pm}\twoheadrightarrow M^{\pm}({\cal{A}})$ 
given by
\[
\phi(x_1) = {\1}_{\llbracket{A}\rrbracket}, \ \
\phi(x_2) = {\1}_{\llbracket{B}\rrbracket}, \ \
\phi(y_1) = {\1}_{\llbracket{O},{A}\rrbracket}, \ \
\phi(y_2) = {\1}_{\llbracket{O},{B}\rrbracket}, \ \ 
\phi(y_3) = {\1}_{\llbracket{A},{B}\rrbracket}, \ \ 
\phi(z) = {\1}_{\llbracket{O},{A},{B}\rrbracket}.
\]
\begin{observation}
\label{obs:Mink-ring-A}
The Minkowski ring of convex sets in the Euclidean plane
that are formed by the intersection of half-planes of the form
from (\ref{eqn:half-planes}) can be presented by $M^{\pm}({\cal{A}})$
and there is a $\comps$-algebra isomorphism
\begin{equation}
\label{eqn:isom}  
M^{\pm}({\cal{A}})\cong {\comps}[x_1,x_2,y_1,y_2,y_3,z]^{\pm}/I
\end{equation}
for some ideal $I$ of ${\comps}[x_1,x_2,y_1,y_2,y_3,z]^{\pm}$.
\end{observation}
The main task of this section is to determine the ideal $I$ in
(\ref{eqn:isom}), thereby obtaining a complete description of all the
Minkowski ring theoretic relations the convex sets in ${\cal{A}}$ satisfy. 

By restricting to one-dimensional closed sets of ${\reals}^2$ we obtain
by Corollary~\ref{cor:d-box-pm} in the previous Section~\ref{sec:simple}
that the analogous elements in
\begin{equation}
\label{eqn:dim-1}
G_1 = \{(y_1 - 1)(y_1 - x_1), (y_2 - 1)(y_2 - x_2), (y_3 - x_1)(y_3 - x_2)\}
\end{equation}
are all contained in the ideal $I$. Note that we can use the elements
in $G_1$ as a system of reductions contained in a Gr\"{o}bner Basis
w.r.t.~LEX
and write each $y_i^2$ in terms of $x_jy_i$ and $x_{\ell}$. This corresponds
to writing each interval of length two as the sum of two intervals of
length one with their common point subtracted.

Each of the mixed products $y_iy_j$ can be reduced further by noting
that each $y_iy_j$, where $i\neq j$, corresponds to a diamond composite
by two triangular regions corresponding to translates of $z$ and $z^{-1}$.
We will describe shortly what exactly constitutes a reduction from
a relation in $I$. But first we give one description of a closed
convex set in ${\cal{A}}$.

Since $O = (0,0)$, ${A} = (1,0)$
and ${B} = (\frac{1}{2},\frac{\sqrt{3}}{2})$, then
any side of a closed convex set $S$
in ${\cal{A}}$ has a slope of $0, \pm \sqrt{3}$ and by starting
at one vertex of $S$ and traversing counterclockwise along the edges, we
see that $S$ is the convex hull of six points
$O_2, A_1, A_2, B_1, B_2, O_1$ in the hexagonal grid in counterclockwise order
where
$\overrightarrow{O_2A_1} = n_1\overrightarrow{OA}$,
$\overrightarrow{A_1A_2} = m_2\overrightarrow{OB}$,
$\overrightarrow{A_2B_1} = n_3\overrightarrow{AB}$,
$\overrightarrow{B_1B_2} = m_1\overrightarrow{AO}$,
$\overrightarrow{B_2O_1} = n_2\overrightarrow{BO}$ and
$\overrightarrow{O_1O_2} = m_3\overrightarrow{BA}$,
for some nonnegative integers $n_1,n_2,n_3,m_1,m_2,m_3$ satisfying
\begin{equation}
\label{eqn:nmN}
n_1 + m_2 + m_3 = m_1 + n_2 + m_3 = m_1 + m_2 + n_3 := N.
\end{equation}
Note that the vertices are here allowed to overlap, for example
$m_i = 0$ for each $i$ will result in a scaled version of the
closed equilateral triangle $\triangle(OAB)$. Hence, each closed set
in ${\cal{A}}$ is therefore obtained by
(i) scaling up $N$-fold the triangle $\triangle(OAB)$,
(ii) capping off $m_3$-fold copy of $\triangle(OAB)$ at the lower left vertex,
(iii) an $m_2$-fold copy of $\triangle(OAB)$ at the lower right vertex,
(iv) an $m_1$-fold copy of $\triangle(OAB)$ at the top vertex, and finally
(v) translate it within the hexagonal lattice.
We note that the entire hexagonal lattice
corresponds to $\{x_1^mx_2^n : m, n \in \ints\}\subseteq
  {\comps}[x_1,x_2,y_1,y_2,y_3,z]^{\pm}/I$, 
so if a closed convex set $S\in M^{\pm}({\cal{A}})$ is presented in terms
of a polynomial $f \in {\comps}[x_1,x_2,y_1,y_2,y_3,z]^{\pm}/I$,
then a translate of $S$ is presented as $x_1^mx_2^nf$ for some
$(m,n)\in {\ints}^2$. This form is indeed unique.
By definition of the homomorphism $\phi$ from above
we have the following.
\begin{proposition}
\label{prp:norm-1}
The preimage of any closed convex set $S$ in ${\cal{A}}$ in the homomorphism
$\phi$ from above has the form $x_1^ax_2^b{\mathbf{c}}_S$ where
\[
  {\mathbf{c}}_S = z^N - (z^{m_3} - y_3^{m_3})
  -x_1^{m_3+n_1}(z^{m_2} - y_2^{m_2})
  -x_2^{m_3+n_2}(z^{m_1} - y_1^{m_1}),
\]
the nonnegative integers $n_1,n_2,n_3,m_1,m_2,m_3$ satisfy
(\ref{eqn:nmN}) and $a,b\in {\ints}$ are integers.
This will be called the {\em first normal form}
of a closed convex set in ${\comps}[x_1,x_2,y_1,y_2,y_3,z]^{\pm}/I$.
\end{proposition}
\begin{example}
\label{exa:rhombus}
The monomial $y_1^2y_2^2$ corresponds to a closed rhombus with vertices
$O, 2A, 2B$ and $2(A+B)$. Here we have in $M^{\pm}({\cal{A}})$ that
\[
y_1^2y_2^2 = z^2 - x_1(z - y_2) - x_2(z - y_1),
\]
so here $N = 2$, $m_1 = m_2 = 1$, $m_3 = 0$, $n_1 = n_2 = 1$ and implicitly
$n_3 = 0$.
\end{example}
\begin{example}
\label{exa:z-1}
The Laurent polynomial
$x_1x_2z^{-1} + x_1y_2 + x_2y_1 + y_3 - x_1 - x_1 - x_1x_2$
corresponds to a closed upside-down triangle with vertices $A, B$ and $A+B$.
Here we have in $M^{\pm}({\cal{A}})$ that
\[
x_1x_2z^{-1} + x_2y_1 + x_1y_2 + y_3 - x_1 - x_2 - x_1x_2 =
z^2 - (z - y_3)  - x_1(z - y_2) - x_2(z - y_1),
\]
so here
$N = 2$, $m_1 = m_2 = m_3 = 1$, $n_1 = n_2 = 0$ and implicitly $n_3 = 0$. 
\end{example}
That the relations from Examples~\ref{exa:rhombus} and~\ref{exa:z-1} hold
in ${\comps}[x_1,x_2,y_1,y_2,y_3,z]^{\pm}/I$ will become apparent in
Theorem~\ref{thm:A}.
\begin{rmk*}
The form for $x_1^ax_2^b{\mathbf{c}}_S$ from Proposition~\ref{prp:norm-1}
is unique and indeed an if and only if statement:
the preimage of a set in ${\cal{A}}$
has the form $x_1^ax_2^b{\mathbf{c}}_S$ if and only if the set
is closed and convex. 
\end{rmk*}
The second normal form, akin to what was presented in
Observation~\ref{obs:dim-tiling}, is solely based on the fact
that each closed convex set formed by our Coxeter arrangement
of lines can be tiled by equilateral triangles in the case
when $d=2$. To facilitate expressions and notations,
we introduce new variables in
terms of those we use that will represent the relative interior of the
faces corresponding to $x_1,x_2,y_1,y_2,y_3$ and $z$ as follows:
\begin{eqnarray*}
  \mathring{y_1} & := & y_1 - 1 - x_1, \\
  \mathring{y_2} & := & y_2 - 1 - x_2, \\
  \mathring{y_3} & := & y_3 - x_1 - x_2, \\
  \mathring{z}   & := & z  - y_1 - y_2 - y_3 + 1 + x_1 + x_2.
\end{eqnarray*}
We note that $\mathring{x_i} = x_i$ and $\mathring{y_i} = -x_iy_i^{-1}$
for $i=1,2$ and $\mathring{y_3} = -x_1x_2y_3^{-1}$.
Since any closed convex set $S \in {\cal{A}}$ can be tiled disjointly
by translates of
$1, \mathring{y_1}, \mathring{y_2}, \mathring{y_3}, \mathring{z}$ and
$z^{-1}$, then we have the following.
\begin{proposition}
\label{prp:norm-2}
The preimage of any closed set which is the union of
closed convex sets in ${\cal{A}}$ in the homomorphism
$\phi$ is a unique disjoint union of translates of the relative
open sets $1, \mathring{y_1}, \mathring{y_2}, \mathring{y_3},
\mathring{z}$ and $z^{-1}$, or equivalently of the relative
open sets $1, y_1^{-1}, y_2^{-1}, y_3^{-1}, \mathring{z}$ and $z^{-1}$.
This will be called the {\em second normal form}
of a closed convex set in ${\comps}[x_1,x_2,y_1,y_2,y_3,z]^{\pm}/I$.
\end{proposition}
Going back to relations that can be used as reductions, we need
to specify more precisely what we mean by a {\em reduction} in
the ring ${\comps}[x_1,x_2,y_1,y_2,y_3,z]^{\pm}$. We mimic
what we said about the one-dimensional elements in $G_1$
from (\ref{eqn:dim-1}), namely, if the area of
$\triangle(OAB)$ is set to one unit, then we can put weight
on each monomial in ${\mathbf{m}}\in [x_1,x_2,y_1,y_2,y_3,z]^{\pm}$ by
letting it be the area of $\phi({\mathbf{m}})\in M^{\pm}({\cal{A}})$.
This area weight is not a partial monomial order, let alone term order.

{\sc Convention:}
In what follows, we will talk about the ``lattice points'' or just
``points'' for the terms $x_1^mx_2^n$ where $m, n \in \ints$,
``intervals'' for $y_1,y_2$ and $y_3$, their powers and translates
and ``triangles, trapezoids, parallelograms'' etc.~ for terms/monomials
that are mapped by $\phi$ to the mentioned polygons.

Looking first at the three $z^2$ relations we obtain by covering
the triangle $z^2$ of area $4$ by a small triangle
and a trapezoid, namely $z + y_3z$, $x_1z + y_2z$ and $x_2z + y_1z$
respectively, we obtain by the I/E principle
$[U\cup V] + [U\cap V] = [U] + [V]$ for sets $U,V\subseteq {\reals}^2$
the following relations:
\[
(z-1)(z-y_3) = (z-x_1)(z-y_2) = (z-x_2)(z-y_1) = 0.
\]
Looking secondly at the three other $z^2$ relations we obtain by covering 
the triangle $z^2$ by two smaller trapezoids, namely
$y_1z + y_2z$, $y_1z + y_3z$ and $y_2z + y_3z$ respectively, we obtain
similarly by the I/E principle the following relations:
\[
(z-y_1)(z-y_2) = (z-y_1)(z-y_3) = (z-y_2)(z-y_3) = 0.
\]
As in (\ref{eqn:dim-1}) we let $G_2$ be the set of the corresponding
elements just obtained from these two-dimensional coverings of
the triangle $z^2$:
\begin{equation}
\label{eqn:dim-2}
\begin{split}
  G_2 = \left\{\right. & \left. (z-1)(z-y_3),(z-x_1)(z-y_2),
  (z-x_2)(z-y_1),\right. \\
  & \left. (z-y_1)(z-y_2), (z-y_1)(z-y_3),
  (z-y_2)(z-y_3)\right.\left.\right\}.
\end{split}
\end{equation}
From the above we have $G_2\subseteq I$ and hence we have the
ideal containment $I' = (G_1\cup G_2)\subseteq I$.
This ideal $I'$ can now be used in our Minkowski ring
$M^{\pm}({\cal{A}})$ to reduce each monomial in
$[x_1,x_2,y_1,y_2,y_3,z]$ to a linear combination of
monomials of the form $x_1^ax_2^b\mathbf{m}$ where
$\mathbf{m}\in E \cup D \cup T$ where
\begin{eqnarray*}
  E & = & \{y_1, y_1, y_3\}, \\
  D & = & \{y_1y_2, y_1y_3,y_2y_3\}, \\
  T & = & \{y_1z, y_2z, y_3z\}.
\end{eqnarray*}
The images of the monomials in $E$ are the edges of the
triangle $\triangle(OAB)$, the images of the monomials of $D$
are the diamonds formed by the Minkowski sum of two of the edges
of $\triangle(OAB)$ and the monomials in $T$ are the trapezoids
formed by the Minkowski sum of the triangle $\triangle(OAB)$ itself
and each of its edges. We will refer to these monomials by their
geometric images.

Consider now all the possible coverings of trapezoids
$y_1z$, $y_2z$ and $y_3z$ of area $3$ from $T$
by two diamonds $y_1y_2$, $y_1y_3$
and $y_2y_3$ from $D$ on one hand, and by a diamond from $D$ and a triangle
$z$, $x_1z$ and $x_2z$ from $T$ on the other. In this
case we obtain no new relations than those induced by $I'$
and hence no new elements in $I$.

Finally, consider the coverings of the diamonds $y_1y_2$,
$y_1y_3$ and $y_2y_3$ of area $2$ from $D$ by two smaller triangles
$z$ and $z^* = \overline{z^{-1}} =
z^{-1} - y_1^{-1} - y_2^{-1} - y_3^{-1} + x_1^{-1} + x_2^{-1} + 1$,
corresponding to the topological closure of the preimage of $z^{-1}$.
Also in this case we do not obtain any new relations than those
induced by $I'$. That is to say, all these geometric cases of how
to cover all these shapes of area at most $3$ yield no new
relations than those already induced by $I' = (G_1\cup G_2)$.
Since all other overlapping coverings of shapes of area of $4$
or more follow from those already mentioned, we conclude that
the ideal $I$ from (\ref{eqn:isom}) is given by $I = (G_1\cup G_2)$
where $G_1$ and $G_2$ are as in (\ref{eqn:dim-1}) and (\ref{eqn:dim-2})
respectively. Hence, we have one of our main results in this section.
\begin{theorem}
\label{thm:A}
The Minkowski ring $M^{\pm}({\cal{A}})$ can be presented as
in (\ref{eqn:isom})
\[
M^{\pm}({\cal{A}})\cong {\comps}[x_1,x_2,y_1,y_2,y_3,z]^{\pm}/I
\]
where $I = (G_1\cup G_2)$ and $G_1$ and $G_2$ are as in
(\ref{eqn:dim-1}) and (\ref{eqn:dim-2}) respectively.
\end{theorem}

We will now use the elements in $I = (G_1 \cup G_2)$ to transform
the preimage of any convex set in in $M^{\pm}({\cal{A}})$ written
in the unique first normal form Proposition~\ref{prp:norm-1} to
the unique second from Proposition~\ref{prp:norm-2} corresponding
to the unique preimage of disjoint points, disjoint relatively open
unit edges and disjoint relatively open unit triangles that disjointly cover
the Euclidean plane ${\reals}^2$.

First note that by Observation~\ref{obs:dim-tiling} we have
that each closed interval can 
be tiled by a disjoint union of the relative open translates
of the sets $1, \mathring{y_1}, \mathring{y_2}, \mathring{y_3}$.
More specifically, (\ref{eqn:y-tiling})
holds for $y = y_i$ and $x = x_i$ for every $n\in\nats$ when $i=1,2$,
and for $y_3$ we have similarly the following.
\begin{equation}
\label{eqn:y3-tiling}
y_3^n = \sum_{i=0}^n x_1^ix_2^{n-i} + \mathring{y_3}\sum_{i=0}^{n-1}x_1^ix_2^{n-1-i}.
\end{equation}
These all follow from the relations presented by the elements of $G_1$ solely.
Therefore it suffices to show that each triangle $z^n$
can be reduced to, or tiled by, the unique second normal form. 

Consider now a triangle $z^n$ of area $n^2$ in
$M^{\pm}({\cal{A}})\cong {\comps}[x_1,x_2,y_1,y_2,y_3,z]^{\pm}/I$.
For each $n\geq 0$ let
\[
f_n := \sum_{i=1}^{n+1}\frac{x_1^i - x_2^i}{x_1-x_2}\in
{\comps}[x_1,x_2,y_1,y_2,y_3,z]^{\pm}/I.
\]
A second normal form or tiling of $z^n$ in terms of translates of
$1, \mathring{y_1}, \mathring{y_2}, \mathring{y_3}, \mathring{z}$ and $z^{-1}$,
as stated in Proposition~\ref{prp:norm-2}, is given by the following claim.
\begin{claim}
\label{clm:Iprime}  
For each $n\in\nats$ we have 
\[
z^n \equiv f_n
+ f_{n-1}(\mathring{y_1} + \mathring{y_2} + \mathring{y_3} + \mathring{z})
+ f_{n-2}x_1x_2z^{-1} \pmod{I}
\]
where $\equiv$ means equality modulo the ideal $I = (G_1\cup G_2)$.
\end{claim}
To prove Claim~\ref{clm:Iprime} we first note that for $n=1$
Claim~\ref{clm:Iprime} becomes $z \equiv (1 + x_1 + x_2) +
(\mathring{y_1} + \mathring{y_2} + \mathring{y_3} + \mathring{z})$ which
trivially holds by mere definitions of the relative open variables
$\mathring{*}$ (and without any use of the relations induced by $I$.)
We proceed by induction on $n$ and assume that Claim~\ref{clm:Iprime}
holds for each integer from $\{1,\ldots,n-1\}$. In that
case it suffices to prove the following.
\begin{lemma}
\label{lmm:strip-tiling}
For any $n\in\nats$ we have the following relation
\begin{equation}
\label{eqn:strip}  
\begin{split}
&  z^n - z^{n-1} \\
&  \equiv \frac{x_1^{n+1} - x_2^{n+1}}{x_1-x_2} +
\left(\frac{x_1^{n} - x_2^{n}}{x_1-x_2}\right)
(\mathring{y_1} + \mathring{y_2} + \mathring{y_3} + \mathring{z})
+ \left(\frac{x_1^{n-1} - x_2^{n-1}}{x_1-x_2}\right)(x_1x_2z^{-1})
\pmod{I}.
\end{split}
\end{equation}
\end{lemma}
\begin{proof}
Since $(z - 1)(z - y_3)\in I$ we have, either geometrically or
directly by induction on $n$, that 
$z^{n} - z^{n-1} \equiv y_3^{n-1}z - y_3^{n-1} \pmod{I}$ and therefore  
(\ref{eqn:strip}) is equivalent to
\begin{equation}
\label{eqn:strip-2}  
\begin{split}
& y_3^{n-1}z - y_3^{n-1} \\
& \equiv \frac{x_1^{n+1} - x_2^{n+1}}{x_1-x_2} +
\left(\frac{x_1^{n} - x_2^{n}}{x_1-x_2}\right)
(\mathring{y_1} + \mathring{y_2} + \mathring{y_3} + \mathring{z})
+ \left(\frac{x_1^{n-1} - x_2^{n-1}}{x_1-x_2}\right)(x_1x_2z^{-1})
\pmod{I}.
\end{split}
\end{equation}
Note that (\ref{eqn:strip-2}) clearly holds for $n=1$ so proceeding
by induction we assume (\ref{eqn:strip-2}) holds for fixed $n\geq 1$.
For the inductive implication from $n$ to $n+1$ it suffices to show that
\begin{equation}
\label{eqn:strip-ed}  
y_3^{n}z - y_3^{n} - x_1(y_3^{n-1}z - y_3^{n-1})
\equiv
x_2^{n+1} + x_2^n(z - 1 - x_1 - x_2) + x_1x_2^nz^{-1} \pmod{I}
\end{equation}
holds for every $n\geq 1$.

Since $(y_3^{n-1} - x_2^{n-1})(y_3 - x_1)(z - 1) \equiv 0 \pmod{I}$
we obtain that left hand side of (\ref{eqn:strip-ed}) satisfies
\[
y_3^{n}z - y_3^{n} - x_1(y_3^{n-1}z - y_3^{n-1}) =
y_3^{n-1}(y_3 - x_1)(z - 1) \equiv
x_2^{n-1}(y_3 - x_1)(z - 1) \pmod{I}
\]
and hence, to prove (\ref{eqn:strip-ed}), it suffices to prove
\[
x_2^{n-1}(y_3 - x_1)(z - 1) \equiv
x_2^{n+1} + x_2^n(z - 1 - x_1 - x_2) + x_1x_2^nz^{-1} \pmod{I},
\]
or, by factoring out $x_2^{n-1}$, it suffices to prove
\begin{equation}
\label{eqn:core}
(y_3 - x_1)(z-1) \equiv x_2(z - 1 - x_1 + x_1z^{-1}) \pmod{I}.
\end{equation}
Since $(z-1)(z-x_1)(z-x_2)\in I$ we obtain that
$(1 - z^{-1})(z-x_1)(z-x_2)\in I$. By first moving the expression on the
left hand side of (\ref{eqn:core}) to its right hand side, then subtracting
$(1 - z^{-1})(z-x_1)(z-x_2)$ we obtain that (\ref{eqn:core})
is equivalent to $(z-1)(z-y_3)\in I$ which clearly holds thereby
completing the proof of the lemma.
\end{proof}

We conclude this subsection by briefly showing that each of the generators
for the ideal $I = (G_1\cup G_2)$ from the above Theorem~\ref{thm:A}
is indeed needed. That is to say,
we cannot obtain the same ideal $I$ by a strict subset of generators
from $G_1\cup G_2$.

Firstly, by letting $y_1 = 2$ and all the other variables of
${\comps}[x_1,x_2,y_1,y_2,y_3,z]^{\pm}$ be $1$, then we get that
the first generator $(y_1 - 1)(y_1 - x_1)$ from $G_1$ equals $1$ and
all the other generators from $G_1\cup G_2$ equal zero. Hence,
$(y_1 - 1)(y_1 - x_1)$ is not in the ideal generated by the other
elements from $G_1\cup G_2$. Similarly, by letting
(i) $y_2 = 2$ and all the other variables be $1$ and 
(ii) $y_3 = 2$ and all the other variables be $1$, we get that
neither of the two remaining generators of $G_1$ can be omitted.

Secondly, by the same token, by letting
(i) $y_1 = y_2 = 1$ and all the remaining variables be $0$,
(ii) $y_1 = y_3 = 1$ and all the remaining variables be $0$ and 
(iii) $y_2 = y_3 = 1$ and all the remaining variables be $0$, we get that
none of the last three generators of $G_2$ can be omitted to generate
the ideal $I = (G_1\cup G_2)$.

Thirdly, and similarly, we consider the following:
(i) By letting $z = y_1 = y_2 = x_1 = x_2 = 2$ we get that
$(z - 1)(z - y_3) = 2 - y_3$, $(y_3 - x_1)(y_3 - x_2) = (2 - y_3)^2$ and
all the other generators from $G_1\cup G_2$ evaluate to zero.
If the first generator $(z - 1)(z - y_3)$ is in the ideal generated by
the other generators from $G_1\cup G_2$, then $2 - y_3$ would be in the
ideal generated by $(2 - y_3)^2$ which is impossible. Hence, the generator
$(z - 1)(z - y_3)$ cannot be omitted from $G_1\cup G_2$.
(ii) By letting $z = y_1 = y_3 = x_2 = 1$ we get that
$(z - x_1)(z - y_2) = (x_1 - 1)(y_2 - 1)$,
$(y_2 - 1)(y_2 - x_2) = (y_2 - 1)^2$ and
all the other generators from $G_1\cup G_2$ evaluate to zero.
As in (i), if the second generator $(z - x_1)(z - y_2)$ is in the
ideal generated by the other generators from $G_1\cup G_2$,
then $(x_1 - 1)(y_2 - 1)$ would be in the
ideal generated by $(y_2 - 1)^2$ which is impossible.
Hence, the generator $(z - x_1)(z - y_2)$ cannot be omitted
from $G_1\cup G_2$ either.
(iii) By symmetry, by letting $z = y_2 = y_3 = x_1 = 1$ we get that
$(z - x_2)(z - y_1) = (x_2 - 1)(y_1 - 1)$,
$(y_1 - 1)(y_1 - x_1) = (y_1 - 1)^2$. Again,
if the third generator $(z - x_2)(z - y_1)$ is in the
ideal generated by the other generators from $G_1\cup G_2$,
then $(x_2 - 1)(y_1 - 1)$ would be in the
ideal generated by $(y_1 - 1)^2$ which is impossible.
Hence, the generator $(z - x_2)(z - y_1)$ cannot be omitted
from $G_1\cup G_2$ either.

We summarize in the following.
\begin{observation}
\label{obs:G1-G2-needed}
The ideal $I = (G_1\cup G_2)$ from Theorem~\ref{thm:A} is minimally
generated by the sets $G_1$ and $G_2$ from
(\ref{eqn:dim-1}) and (\ref{eqn:dim-2}) respectively; no proper subset
of generators in $G_1\cup G_2$ generates $I$.
\end{observation}

\subsection{Minkowski identities}
Note that each element in $G_2$ has the form
$\prod_{F\in {\cal{P}}(P)}([P] - [F])$ where ${\cal{P}}(P)$ is a
set of faces of $P$ that cover the vertices of $P$, that is,
for every vertex $v$ of $P$ there is an $F\in {\cal{P}}(P)$
with $v\in F$. This holds in general too as the following proposition
states, and it follows from a slightly more general
statement~\cite[Corollary~2.5]{Jay-Klaus}.
\begin{proposition}
\label{prp:cover}
Let $P$ be a convex polytope in $W = {\reals}^d$
with (a finite) vertex set $V$.
Then for any set ${\cal{P}}(P)$ of faces of $P$ that cover $V$, so
$V = \bigcup_{F\in {\cal{P}}(P)}F$, we have
$\prod_{F\in {\cal{P}}(P)}([P] - [F]) = 0$ in the Minkowski ring
$M({\cal{P}}(W))$.
\end{proposition}
\begin{rmk*}
A special case of the above Proposition~\ref{prp:cover}, where
${\cal{C}}(P)$ consists of all the vertices of the polytope $P$,
appears as Theorem 5 in~\cite[p.~12]{Lawrence} and its proof
depends on Theorem 3 from~\cite[p.~9]{Lawrence}. This Theorem 3, however,
can also be used to prove the above Proposition~\ref{prp:cover}
directly.
\end{rmk*}  
Working in ${\comps}[x_1,x_2,y_1,y_2,y_3,z]^{\pm}/I$ we obtain by
direct manipulation
\begin{eqnarray*}
  (z-1)(z-x_1)(z-x_2) & = & (z-y_1)(z-x_1)(z-x_2) + (y_1-1)(z-x_1)(z-x_2) \\
  & = & (y_1-1)(z-x_1)(z-x_2) \\ 
  & = & (y_1-1)(z-y_1)(z-x_2) + (y_1-1)(y_1-x_1)(z-x_2) \\
  & = & (y_1-1)[(z-y_1)(z-x_2)] + [(y_1-1)(y_1-x_1)](z-x_2) \\
  & = & 0,
\end{eqnarray*}
and hence the known identity $\prod_{v\in P}([P] - [v]) = 0$ from
from~\cite[Theorem 5, p.~12]{Lawrence} for the triangle $z$ does not constitute
a new relation from those induced by $I = (G_1\cup G_2)$.
This observation can be generalized but first we need a definition.
\begin{definition}
\label{def:MCFP}
For a $d$-polytope $P$ and $\ell\leq d$ let ${\cal{F}}_{\ell}(P)$
denote all the closed
$\ell$-faces of $P$ and let ${\cal{F}}(P)$ denote the set of all proper
faces of $P$, so ${\cal{F}}(P) = \bigcup_{\ell = 0}^d{\cal{F}}_{\ell}(P)$.
The \emph{Minkowski face ring of $P$} is the
Minkowski subring $M({\cal{F}(P)})\subseteq M({\cal{P}}(W))$ generated by
the finite subset ${\cal{F}}(P)$ of closed convex sets of $W = {\reals}^d$.

Likewise, if the indicator function of each face has a multiplicative inverse
then we denote the Minkowski face ring by
$M^{\pm}({\cal{F}}(P))$.
\end{definition}
To best describe this generalization we need some additional terminologies.

For our polytope $P$ let $\tilde{\cal{C}}_{\ell}(P)$ denote the collection of
all sets of faces that cover all the $\ell$-faces of $P$, that is to say
\[
\tilde{\cal{C}}_{\ell}(P) = \{\cal{C}\subseteq \cal{F}(P) :
X\in {\cal{F}}_{\ell}(P)\Rightarrow X\subseteq Y\mbox{ for some }
Y\in{\cal{C}}\}.
\]
For $\cal{C}\in\tilde{\cal{C}}_0(P)$ we have by
Proposition~\ref{prp:cover} that the identity
\begin{equation}
\label{eqn:ID}
\id(P;\cal{C}) : \prod_{F\in{\cal{C}}}([P] - [F]) = 0
\end{equation}
holds in the Minkowski ring $M({\cal{P}}(W))$. We now note a few properties:
(i) if $\cal{C}\subseteq\cal{C}'$ for two elements in $\tilde{\cal{C}}_0(P)$,
then clearly the identity $\id(P;\cal{C})$ implies the identity
$\id(P;\cal{C}')$.
(ii) We can always assume $\cal{C}\in\tilde{\cal{C}}_0(P)$ to be an
antichain w.r.t.~inclusion, since if $A,B\in\cal{C}$ and $A\subseteq B$,
then $\cal{C}\setminus\{A\}\in \tilde{\cal{C}}_0(P)$ and the identify
$\id(P;\cal{C}\setminus\{A\})$ holds in $M({\cal{P}}(W))$ and it implies
the identity $\id(P;\cal{C})$. This means that in our search for defining
identities $\id(P;\cal{C})$ among all possible
$\cal{C}\in\tilde{\cal{C}}_0(P)$, i.e.~those that imply other identities of this
same type (or rather those elements that generate the ideal when
presenting the Minkowski ring as a quotient of a polynomial ring and
an ideal), we can restrict our search to those identities $\id(P;\cal{C})$
where $\cal{C}$ is minimal w.r.t.~inclusion. By (ii) these minimal
sets are then necessarily antichains. Let
$\tilde{\cal{A}}_0(P)\subseteq \tilde{\cal{C}}_0(P)$ denote the set of these
antichains.

Suppose $\cal{A},\cal{A}'\in \tilde{\cal{A}}_0(P)$. We say that
$\cal{A}'$ {\em covers} $\cal{A}$ if there is a face $A\in\cal{A}$ of $P$
and a set $\cal{B}\in \tilde{\cal{C}}_0(A)$ of faces of $A$ that covers
all the vertices of $A$ such that
${\cal{A}}' = (\cal{A}\setminus\{A\})\cup\cal{B}$.
In this case we have the following
\begin{proposition}
\label{prp:cover-id}
If $P$ is a $d$-polytope, $\cal{A},\cal{A}'\in \tilde{\cal{A}}_0(P)$ are
antichains of faces of $P$ such that $\cal{A}'$ covers $\cal{A}$, then
$\id(P;\cal{A})$ implies the identity $\id(P;\cal{A}')$ in the Minkowski
ring $M({\cal{P}}(W))$.
\end{proposition}
\begin{proof}
We will prove this by induction on $d$, the dimension of the polytope $P$.
Being trivially true for $d=1$, we assume $\id(P;\cal{A})$ holds and
show that $\id(P;\cal{A}')$ will then hold as well.
Since $\cal{A}'$ covers $\cal{A}$
there is a face $A\in\cal{A}$ and a cover $\cal{B}\in \tilde{\cal{C}}_0(A)$ of 
the vertices of $A$ such that
$\cal{A}' = (\cal{A}\setminus\{A\})\cup\cal{B}$. In this case we obtain 
\[
\prod_{F\in{\cal{A}'}}([P] - [F])
= \prod_{F\in(\cal{A}\setminus\{A\})\cup\cal{B}}([P] - [F]) 
= \prod_{F\in\cal{A}\setminus\{A\}}([P] - [F])\cdot
\prod_{F\in\cal{B}}([P] - [F]).
\]
Letting $\pi_A := \prod_{F\in\cal{A}\setminus\{A\}}([P] - [F])$ and
$m = |\cal{B}|$ we then get by direct manipulation that 
\begin{eqnarray*}
\prod_{F\in{\cal{A}'}}([P] - [F])
& = & \pi_A\prod_{F\in\cal{B}}([P] - [F]) \\
& = & \pi_A\prod_{F\in\cal{B}}([P] - [A] + [A] - [F]) \\
& = & \pi_A\sum_{i=0}^m\left( ([P] - [A])^i
\sum_{\substack{{\cal{S}}\subseteq{\cal{B}}\\|\cal{S}| = m-i}}
\prod_{F\in\cal{S}}([A] - [F])\right) \\
& = & \pi_A\prod_{F\in\cal{B}}([A] - [F])\\
&   & + \pi_A([P] - [A])\sum_{i=1}^m\left( ([P] - [A])^{i-1}
\sum_{\substack{{\cal{S}}\subseteq{\cal{B}}\\|\cal{S}| = m-i}}
\prod_{F\in\cal{S}}([A] - [F])\right).
\end{eqnarray*}
Since by assumption we have
$\pi_A([P] - [A]) = \prod_{F\in\cal{A}}([P] - [F]) = 0$, 
and since the dimension of $A$ is less than $d$, the dimension of $P$,
then we have by induction hypothesis that the identity
$\id(A;\cal{B}) : \prod_{F\in{\cal{B}}}([A] - [F]) = 0$ also holds
in $M({\cal{P}}(W))$ and hence we obtain from the above display that
\[
\prod_{F\in{\cal{A}'}}([P] - [F]) = \pi_A\prod_{F\in\cal{B}}([A] - [F])
+ 0\cdot \sum_{i=1}^m\left( ([P] - [A])^{i-1}
\sum_{\substack{{\cal{S}}\subseteq{\cal{B}}\\|\cal{S}| = m-i}}
\prod_{F\in\cal{S}}([A] - [F])\right) = 0
\]
which proves the proposition.
\end{proof}
We can define a partial order $\preceq$ on $\tilde{\cal{A}}_0(P)$ by
letting $\cal{A}\preceq\cal{A}'$ if there is a finite sequence
$\cal{A} = \cal{A}_0, \cal{A}_1,\ldots,\cal{A}_k = \cal{A}'$ such that
for each $\cal{A}_{i+1}$ covers $\cal{A}_i$. Since the empty sequence
where $k= 0$ is allowed and for each $i$ we have
\[
|\cal{A}_{i+1}| = |(\cal{A}_i\setminus A)\cup\cal{B}|
= |\cal{A}_i| - 1 + |\cal{B}| \geq |\cal{A}_i|,
\]
then $\preceq$ is indeed a partial order. By the above
Proposition~\ref{prp:cover-id} we obtain the following corollary.
\begin{corollary}
\label{cor:minimal}
If $P$ is a $d$-polytope, then all the identities $\id(P;\cal{C})$
where $\cal{C}\in\tilde{\cal{C}}_0(P)$ are implied by the identities
$\id(P;\cal{C})$ where $\cal{C}$ is a minimal element of the poset
$(\tilde{\cal{A}}_0(P),\preceq)$.
\end{corollary}
Note that the identity $\id(P;\cal{F}_0(P)) = \id(P;V(P))$
is only a minimal element
of $(\tilde{\cal{A}}_0(P),\preceq)$ when $d = 1$.
In all other cases for $d\geq 2$, it always follows from other ``shorter''
(i.e.~of smaller degree in $[P]$ as a polynomial) identities.
\begin{corollary}
\label{cor:only-dim-1}
If $P$ is a $d$-polytope and $V(P)$ is the set of vertices of $P$, then
the identity $\id(P;\cal{F}_0(P))$:
\[
\prod_{v\in V(P)}([P] - [v]) = \prod_{F\in\cal{F}_0(P)}([P] - [F]) = 0
\]
is only a defining identity when $d=1$ and the polytope $P$ is a closed
interval.
\end{corollary}
\begin{rmk*}
We note that $(z-1)(z-x_1)(z-x_2)$ is contained in
the ideal $I$ generated by $G_1\cup G_2$ from 
(\ref{eqn:dim-1}) and (\ref{eqn:dim-2}) respectively, consistent
with Corollaries~\ref{cor:minimal} and~\ref{cor:only-dim-1}.
\end{rmk*}
Let $P$ be a $d$-polytope and ${\cal{C}} = {\cal{C}}(P)$ a set of
faces of $P$. Corollary~\ref{cor:only-dim-1} naturally yields the
question of when exactly an identity $\id(P;{\cal{C}}(P)$ holds.
We will now show that in order for $\id(P;{\cal{C}})$ to hold in
the Minkowski ring $M({\cal{F}}(P))$, then ${\cal{C}}$ must cover
all the vertices of $P$. 

Suppose ${\cal{C}} = \{F_1,\ldots,F_k\}$ does not cover all the vertices
of $P$, so there is a vertex $A \in V(P) = {\cal{F}}_0(P)$ 
such that $A \not\in \cv(F_1,\ldots,F_k)$. Suppose $\id(P;{\cal{C}})$ is
an identity that holds in $M({\cal{F}}(P))$. In that case we have
\begin{eqnarray*}
  0 & = & \prod_{F\in {\cal{C}}}([P] - [F]) \\
    & = & \prod_{i=1}^k([P] - [F_i]) \\
    & = & [P]^k - \sum_{i=1}^k[P]^{k-1}[F_i]
+ \sum_{i<j}[P]^{k-2}[F_i][F_j] - + \cdots
+ (-1)^k [F_1]\cdots[F_k].
\end{eqnarray*}
By the definition of multiplication in a Minkowski ring we then have
\begin{equation}
\label{eqn:indicator}  
0 = [kP] - \sum_{i=1}^k[(k-1)P + F_i]
+ \sum_{i<j}[(k-2)P + F_i + F_j] - + \cdots
+ (-1)^k [F_1 + \cdots + F_k].
\end{equation}
Note that the above indicator function holds when evaluated
at every point in ${\reals}^d$, in particular at the point $kA\in {\reals}^d$
which is a vertex of $kP$. We first note that $[kP](kA)  = 1$.
\begin{claim}
\label{clm:kA-not-in}
For any $\ell\in \{1,\ldots, k\}$ and faces
$F_{i_1},\ldots, F_{i_{\ell}}$  we have
$kA\not\in (k-\ell)P + F_{i_1} + \cdots + F_{i_{\ell}}$.
\end{claim}
\begin{proof}
Since $A \not\in \cv(F_1,\ldots,F_k) := Q$, there is a hyperplane $H$ with a
normal vector $\tilde{a}$ separating $A$ from $Q$. Here $H$ can be taken as a
translate of any supporting hyperplane of the vertex $A\in P$, so the function
$\tilde{x}\mapsto\tilde{a}\cdot\tilde{x}$ restricted to $P$ is maximized at
$\tilde{x} = \overrightarrow{OA}\in {\reals}^d$.
Since for any point $B_j\in F_{i_j}$
we have
$\tilde{a}\cdot\overrightarrow{OB_j} < \tilde{a}\cdot\overrightarrow{OA}$
we also have for a general point
$T = T_1 + \cdots + T_{k-\ell} + B_1 + \cdots + B_{\ell}$
of $(k-\ell)P + F_{i_1} + \cdots + F_{i_{\ell}}$ that 
\[
\tilde{a}\cdot\overrightarrow{OT} = \tilde{a}\cdot\overrightarrow{OT_1}
+ \cdot + \tilde{a}\cdot\overrightarrow{OT_{k-\ell}}
+ \tilde{a}\cdot\overrightarrow{OB_1}
+ \cdot + \tilde{a}\cdot\overrightarrow{OB_{\ell}}
< k(\tilde{a}\cdot\overrightarrow{OA})
= \tilde{a}\cdot\overrightarrow{O\/kA},
\]
and so $kA\not\in (k-\ell)P + F_{i_1} + \cdots + F_{i_{\ell}}$.
\end{proof}

By Claim~\ref{clm:kA-not-in} we have that
$[(k-\ell)P + F_{i_1} + \cdots + F_{i_{\ell}}](kA) = 0$ for each
$\ell\in\{1,\ldots,k\}$ and hence from (\ref{eqn:indicator}) when evaluated
at $kA$ we obtain
\[
\begin{split}
0 & = \\
  & \left([kP] - \sum_{i=1}^k[(k-1)P + F_i]
    + \sum_{i<j}[(k-2)P + F_i + F_j] - + \cdots
    + (-1)^k [F_1 + \cdots + F_k]\right)(kA) \\
  & = [kP](kA) \\
  & - \sum_{i=1}^k[(k-1)P + F_i](kA)
    + \sum_{i<j}[(k-2)P + F_i + F_j](kA)
    - + \cdots + (-1)^k [F_1 + \cdots + F_k](kA) \\
  & = 1 \\
  & - \sum_{i=1}^k0
    + \sum_{i<j}0
    - + \cdots + (-1)^k0 \\
  & = 1,  
\end{split}
\]
a blatant contradiction. We therefore have the following.
\begin{proposition}
\label{prp:indicator-iff}
If $P$ is a $d$-polytope and ${\cal{C}} = {\cal{C}}(P)$ a set of
faces of $P$, then $\id(P;{\cal{C}})$ is an identity in
the Minkowski ring $M({\cal{F}}(P))$ if and only if
${\cal{C}}$ covers all the vertices of $P$.
\end{proposition}

\section{The Cartesian product of polytopes}
\label{sec:Cartesian}

There is a well-known and useful homomorphism,
$\bar \chi : {\cal S} \rightarrow \ints$,
mapping the additive group generated by
the indicator functions of closed, convex polyhedra to $\ints$,
namely, the {\em Euler homomorphism} (see~\cite{Hadwiger, Beck-and-Robins, Barvinok, Lawrence-(val+pol), Lawrence-(Euler), Lawrence, Groemer}).
Its value is $1$ on the indicator function of any nomempty, closed, convex
polyhedron, and this property characterizes it. If $f \in \cal S$
is the indicator function of a compact polyhedron then $\bar \chi(f)$
is the usual Euler characteristic of the polyhedron.  Application of this
homomorphism yields the following proposition.
\begin{proposition}
\label{prp:n=m}
If $C_1,\ldots,C_n, C_1',\ldots,C_m'$ are nonempty closed convex polyhedra
in $W = {\reals}^d$ such that $\sum_{i=1}^n [C_i] = \sum_{j=1}^m[C_j']$,
then $m = n$.
\end{proposition}
Consider now the Minkowski ring $M^{\pm}(P_1,\ldots,P_n)$ of a finite set of
polytopes in $W = {\reals}^d$, for example ${\cal{F}}(P)$
the set of all the nonempty faces of a given
$d$-dimensional polytope $P$,
including $P$ itself.
As a ${\comps}$-algebra we have
$M^{\pm}(P_1,\ldots,P_n) \cong {\comps}[x_1,\ldots,x_n]^{\pm}/I$, where
the Minkowski ideal $I$ is kernel of the 
surjection
$\phi: {\comps}[x_1,\ldots,x_n]^{\pm}
\twoheadrightarrow M^{\pm}(P_1,\ldots,P_n)$.
Since $I$ is an extended ideal, obtained by tensoring with ${\comps}$,
the ideal $I$ is generated by finitely many elements from
${\ints}[x_1,\ldots,x_n]$. Assume
$f = \sum_{i=1}^Nc_i\tilde{x}^{\tilde{\alpha}_i}\in I$ where
$\tilde{x}^{\tilde{\alpha}_i} = x_1^{\alpha_{i\/1}}\cdots x_n^{\alpha_{i\/n}}$ and each
$c_i\in {\ints}$. We can therefore write
\begin{equation}
\label{eqn:diff-bc}  
f = \sum_{i=1}^Na_i\tilde{x}^{\tilde{\alpha}_i} -
\sum_{i=1}^Nb_i\tilde{x}^{\tilde{\beta}_i},
\end{equation}
where each $a_i, b_i\in {\nats}\cup \{0\} =\{0,1,2,3,\ldots\}$. Since
$\phi(\tilde{x}^{\tilde{\alpha}})
= [\alpha_{1}P_1 + \cdots + \alpha_{n}P_n]$
as a Minkowski sum then, 
since $f \in I = \ker(\phi)$, we get by (\ref{eqn:diff-bc}) the following
identity in $M^{\pm}(P_1,\ldots,P_n)$
\[
0 = \phi(f) = \sum_{i=1}^Na_i[\alpha_{i\/1}P_1 + \cdots + \alpha_{i\/n}P_n]
- \sum_{i=1}^Nb_i[\alpha_{i\/1}P_1 + \cdots + \alpha_{i\/n}P_n].
\]
Since each polytope $P_i$ is a nonempty closed convex polyhdron, and also
their Minkowski sums, then by Proposition~\ref{prp:n=m} we get that
$\sum_{i=1}^Na_i = \sum_{i=1}^Nb_i$ and hence we have the following.
\begin{theorem}
\label{thm:eval-1}
If $P_1,\ldots,P_n$ are polytopes in $W = {\reals}^d$ and
$M^{\pm}(P_1,\ldots,P_n) \cong {\comps}[x_1,\ldots,x_n]^{\pm}/I$ as
$\comps$-algebras, then for any $f\in I$ we have $f(\tilde{1}) = 0$,
where $\tilde{1} = (1,1,\ldots,1)$.
\end{theorem}
Recall an important corollary of
Hilbert's Nullstellensatz~\cite[p.~85]{Atiyah-Macdonald} 
and~\cite[p.~410]{Aluffi} as the following~\cite[Cor.~2.10, p.~406]{Aluffi}.
\begin{theorem}
\label{thm:HilbertNSS-cor}
If $K$ is an algebraically closed field, then an ideal of 
$K[x_1,\ldots,x_n]$ is maximal if and only if it has the form 
$\m_{\tilde{a}} = (x_1-a_1,\ldots,x_n-a_n)$ for some $\tilde{a}\in K^n$.
\end{theorem}
Since $\comps$ is algebraically closed and the maximal ideals
of ${\comps}[x_1,\ldots,x_n]^{\pm} = M^{-1}{\comps}[x_1,\ldots,x_n]$,
where $M$ is the multiplicatively closed set
$M = \{x_1^{\alpha_1}\cdots x_n^{\alpha_n} : \alpha_1,\ldots,\alpha_n\in\ints\}$,
are in bijective correspondence
with the maximal ideals of ${\comps}[x_1,\ldots,x_n]$ that don't meet $M$,
then each maximal ideal of ${\comps}[x_1,\ldots,x_n]^{\pm}$ has the form
${\m}_{\tilde{a}}$ where each $a_i\neq 0$. From Theorem~\ref{thm:eval-1}
we have the following corollary.
\begin{corollary}
\label{cor:I-in-m1}
If $M^{\pm}(P_1,\ldots,P_n) \cong {\comps}[x_1,\ldots,x_n]^{\pm}/I$ as in
Theorem~\ref{thm:eval-1}, then
$I\subseteq {\m}_{\tilde{1}} = (x_1-1,x_2-1,\ldots, x_n-1)$.
\end{corollary}

Now consider two polytopes $P\subseteq {\reals}^d$ and
$Q\subseteq {\reals}^e$ and their Cartesian product $P\times Q$
in ${\reals}^{d+e}$. If the set of nonempty closed faces of $P$ is
${\cal{F}}(P) = \{A_1,\ldots, A_m\}$ and the set of nonempty faces of
$Q$ is ${\cal{F}}(Q) = \{B_1,\ldots,B_n\}$, then the set of closed nonempty
faces of $P\times Q$ is ${\cal{F}}(P\times Q) =
\{A_i\times B_j : 1\leq i\leq m,\ \ 1\leq j\leq n\}$. Our goal in the rest
of this section is to describe
the Minkowski ring $M^{\pm}({\cal{F}}(P\times Q))$. 

For general sets $A, A' \subseteq {\reals}^d$ and $B, B'\subseteq {\reals}^e$
we have
\begin{equation}
\label{eqn:AABB}
(A\times B) + (A'\times B') = (A + A')\times (B + B')
\end{equation}
as subsets of ${\reals}^{d+e}$. With the identification
$A\cong A\times \{\tilde{0}\}\subseteq {\reals}^{d+e}$ and
$B\cong \{\tilde{0}\}\times B \subseteq {\reals}^{d+e}$, we note
that $A\times B = (A\times \{\tilde{0}\}) + (\{\tilde{0}\}\times B)$
as a Minkowski sum of sets and so in view of this identification
we see that the identify (\ref{eqn:AABB}) reduces to the obvious identity 
$(A + B) + (A' + B') =  (A + A') + (B + B')$ as a Minkowski sum
in ${\reals}^{d+e}$. So, the Cartesian product is just a special
case of a Minkowski sum of sets. This we consider the main reason
why Minkowski rings behave well under taking Cartesian products of
polytopes as we will new see.

Suppose now we have the isomorphisms
\[
M^{\pm}({\cal{F}}(P)) \cong {\comps}[x_1,\ldots,x_m]^{\pm}/I, \ \
M^{\pm}({\cal{F}}(Q)) \cong {\comps}[y_1,\ldots,y_n]^{\pm}/J,
\]
induced by the ${\comps}$-algebra surjections
$\phi_P : {\comps}[x_1,\ldots,x_m]^{\pm}
\twoheadrightarrow M^{\pm}({\cal{F}}(P))$
and 
$\phi_Q : {\comps}[y_1,\ldots,y_n]^{\pm}
\twoheadrightarrow M^{\pm}({\cal{F}}(Q))$
determined by $\phi_P(x_i) = [A_i]$ for each $i$ and
$\phi_Q(y_j) = [B_j]$ for each $j$ respectively, so $I = \ker(\phi_P)$
and $J = \ker(\phi_Q)$. Since for a general face $A_i\times B_j$
of $P\times Q$ we have by the above mentioned identification
that $A_i\times B_j = A_i + B_j$, it corresponds to $x_iy_j$
in ${\comps}[x_1,\ldots,x_m,y_1,\ldots,y_n]^{\pm}$. Therefore
we have an isomorphism
\begin{equation}
\label{eqn:phi}  
M^{\pm}({\cal{F}}(P\times Q))
  \cong {\comps}[x_1,\ldots,x_m,y_1,\ldots,y_n]^{\pm}/L,
\end{equation}
for some ideal $L$, induced by the surjection
$\phi : {\comps}[x_1,\ldots,x_m,y_1,\ldots,y_n]^{\pm}
\twoheadrightarrow M^{\pm}({\cal{F}}(P\times Q))$
determined by $\phi(x_i) = [A_i\times \{\tilde{0}\}]$ for each $i$ and
$\phi(y_j) = [\{\tilde{0}\}\times B_j]$ for each $j$, so
$L = \ker(\phi)$. The natural 
${\comps}$-algebra surjection
\[
{\comps}[x_1,\ldots,x_m]^{\pm}\otimes {\comps}[y_1,\ldots,y_n]^{\pm}
\twoheadrightarrow
{\comps}[x_1,\ldots,x_m,y_1,\ldots,y_n]^{\pm}/L  
\]
yields a ${\comps}$-algebra surjection
\begin{equation}
\label{eqn:tau}
\tau : {\comps}[x_1,\ldots,x_m]^{\pm}/I\otimes {\comps}[y_1,\ldots,y_n]^{\pm}/J
\twoheadrightarrow
{\comps}[x_1,\ldots,x_m,y_1,\ldots,y_n]^{\pm}/L,  
\end{equation}
since the ideal generated by $I$ and $J$ is contained in $L$, that is
$(I,J)\subseteq L$, and so we have a ${\comps}$-algebra
surjection $M^{\pm}({\cal{F}}(P))\otimes M^{\pm}({\cal{F}}(Q))
    \twoheadrightarrow M^{\pm}({\cal{F}}(P\times Q))$.
We want to show that this is a ${\comps}$-algebra isomorphism.
It suffices to show that the surjection in (\ref{eqn:tau}) is injective.
Suppose $f\in L = \ker(\phi)$ from (\ref{eqn:phi}). Since
$f\in {\comps}[x_1,\ldots,x_m,y_1,\ldots,y_n]^{\pm}$, we can write
$f = \sum_{i=1}^Nc_i{\tilde{x}}^{\tilde{\alpha_i}}{\tilde{y}}^{\tilde{\beta_i}}$
where $c_i\in\ints$ for each $i$. That $f\in L$ means exactly that
\[
\begin{split}
0 & = \phi(f) \\
  & = \sum_{i=1}^Nc_i [(\alpha_{i\/1}A_1 + \cdots + \alpha_{i\/m}A_m)\times
  (\beta_{i\/1}B_1 + \cdots + \beta_{i\/n}B_m)] \\
  & = \sum_{i=1}^Nc_i
[\alpha_{i\/1}A_1 + \cdots + \alpha_{i\/m}A_m]
[\beta_{i\/1}B_1 + \cdots + \beta_{i\/n}B_m],
\end{split}
\]
where
$([X][Y])(\tilde{a},\tilde{b}) := ([X](\tilde{a}))([Y](\tilde{b}))$. 
This means that in $M^{\pm}({\cal{F}}(P\times Q))$ we have
for all $\tilde{a}\in{\reals}^d$ and $\tilde{b}\in{\reals}^e$ that
\[
0 =
(\phi(f))(\tilde{a},\tilde{b}) =
\sum_{i=1}^Nc_i
[\alpha_{i\/1}A_1 + \cdots + \alpha_{i\/m}A_m](\tilde{a})
[\beta_{i\/1}B_1 + \cdots + \beta_{i\/n}B_m](\tilde{b}).
\]
In particular, fixing $\tilde{a}\in{\reals}^d$ then
the function $\tilde{b} \mapsto (\phi(f))(\tilde{a},\tilde{b})$
is the zero function. Since
$[\alpha_{i\/1}A_1 + \cdots + \alpha_{i\/m}A_m](\tilde{a})\in\{0,1\}
\subseteq\ints$
we have that
\[
f_{\tilde{a}}(\tilde{y}) =
\sum_{i=1}^N
c_i[\alpha_{i\/1}A_1 + \cdots + \alpha_{i\/m}A_m](\tilde{a})
{\tilde{y}}^{\tilde{\beta_i}} \in J.
\]
By Theorem~\ref{thm:eval-1} we have $f_{\tilde{a}}(\tilde{1}) = 0$.
Writing this very identity out we have
\[
0
= f_{\tilde{a}}(\tilde{1})
= \sum_{i=1}^N c_i[\alpha_{i\/1}A_1 + \cdots + \alpha_{i\/m}A_m](\tilde{a})
= \left(
\sum_{i=1}^Nc_i[\alpha_{i\/1}A_1 + \cdots + \alpha_{i\/m}A_m]\right)(\tilde{a})
\]
for every $\tilde{a}\in {\reals}^d$. By definition we then
have
$f(\tilde{x},\tilde{1}) = \sum_{i=1}^Nc_i{\tilde{x}}^{\tilde{\alpha_i}}\in I$.
By the same token we have 
$f(\tilde{1},\tilde{y}) = \sum_{i=1}^Nc_i{\tilde{y}}^{\tilde{\beta_i}}\in J$.
Therefore we have that if $f = f(\tilde{x},\tilde{y})\in L$,
then $f(\tilde{x},\tilde{1})\in I$ and $f(\tilde{1},\tilde{y})\in J$.
This implies that the ${\comps}$-algebra surjection 
\[
\begin{split}
  {\comps}[x_1,\ldots,x_m,y_1,\ldots,y_n]^{\pm}
  & \cong {\comps}[x_1,\ldots,x_m]^{\pm}\otimes {\comps}[y_1,\ldots,y_n]^{\pm} \\
  & \twoheadrightarrow
  ({\comps}[x_1,\ldots,x_m]^{\pm}/I)\otimes({\comps}[y_1,\ldots,y_n]^{\pm}/J)
\end{split}
\]
given by $f \mapsto \bar{f}(\tilde{x},\tilde{1})\otimes
\bar{f}(\tilde{1},\tilde{y})$ yields a ${\comps}$-algebra
homomorphism (a surjection)
\[
\psi : {\comps}[x_1,\ldots,x_m,y_1,\ldots,y_n]^{\pm}/L \rightarrow 
({\comps}[x_1,\ldots,x_m]^{\pm}/I)\otimes
({\comps}[y_1,\ldots,y_n]^{\pm}/J)
\]
given by $\bar{f} \mapsto \bar{f}(\tilde{x},\tilde{1})\otimes
\bar{f}(\tilde{1},\tilde{y})$. If $\tau$ is the surjection in (\ref{eqn:tau}),
then we have
\[
(\psi\circ\tau)({\tilde{x}}^{\tilde{\alpha}}\otimes{\tilde{y}}^{\tilde{\beta}})
=
\psi({\tilde{x}}^{\tilde{\alpha}}{\tilde{y}}^{\tilde{\beta}})
=
({\tilde{x}}^{\tilde{\alpha}}\cdot 1)\otimes(1\cdot {\tilde{y}}^{\tilde{\beta}})
=
{\tilde{x}}^{\tilde{\alpha}}\otimes{\tilde{y}}^{\tilde{\beta}}
\]
and so $\psi\circ\tau$ is the identity homomorphism on
$({\comps}[x_1,\ldots,x_m]^{\pm}/I)\otimes({\comps}[y_1,\ldots,y_n]^{\pm}/J)$.
This implies that the surjection $\tau$ 
is injective and hence a ${\comps}$-algebra isomorphism.
We therefore have the main theorem
in this sections.
\begin{theorem}
\label{thm:Cart-isom}
For polytopes $P\subseteq {\reals}^d$ and $Q\subseteq {\reals}^e$
we have for the Minkowski ring of all the nonempty faces of $P\times Q$
that
$M^{\pm}({\cal{F}}(P\times Q))
\cong M^{\pm}({\cal{F}}(P))\otimes M^{\pm}({\cal{F}}(Q))$
as ${\comps}$-algebras.
\end{theorem}
By Theorem~\ref{thm:Cart-isom} we get
Corollaries~\ref{cor:d-box} and~\ref{cor:d-box-pm} directly
without the use of any Gr\"{o}bner Bases reductions.
Also, since the simplest form of one dimensional simplex $\Delta_1$
is homeomorphic to the closed interval $[0,1]$, then
${\cal{F}}(\Delta_1) = \{\{0\}, \{1\}, [0,1]\}$ and so
$M^{\pm}({\cal{F}}(\Delta_1)) = M({\cal{B}}_1)$, the Minkowski ring
of the one dimensional box. From Corollary~\ref{cor:d-box-pm}
and Theorem~\ref{thm:Cart-isom} we then obtain the following.
\begin{corollary}
\label{cor:prism}
For a polytope $P\subseteq {\reals}^d$ the Minkowski ring of all the
nonempty faces of the prism $P\times \Delta_1$ is given by
\[
\begin{split}
M^{\pm}({\cal{F}}(P\times\Delta_1))
  & \cong M^{\pm}({\cal{F}}(P))\otimes M^{\pm}({\cal{F}}(\Delta_1)) \\
  & \cong M^{\pm}({\cal{F}}(P))\otimes M({\cal{B}}_1) \\
  & \cong M^{\pm}({\cal{F}}(P))\otimes {\comps}[x,y]^{\pm}/((y-1)(y-x))
\end{split}
\]
as ${\comps}$-algebras.  
\end{corollary}
{\sc Example:} We note that the points $O, A, B$ in
Section~\ref{sec:Coxeter} form a triangle and that the
Minkowski ring
$M^{\pm}({\cal{A}})
= M^{\pm}(\llbracket{A}\rrbracket, \llbracket{B}\rrbracket,
\llbracket{O}, {A}\rrbracket, \llbracket{O},{B}\rrbracket,
\llbracket{A},{B}\rrbracket, \llbracket{O},{A},{B}\rrbracket\})$
is the Minkowski ring of the simples form of the two
dimensional simplex $\Delta_2$ together with
all its faces, so $M^{\pm}({\cal{A}}) = M^{\pm}({\cal{F}}(\Delta_2))$.
By Theorems~\ref{thm:A} and~\ref{thm:Cart-isom} we then have
\[
\begin{split}
M^{\pm}({\cal{F}}(\Delta_1\times\Delta_2))
  & \cong M^{\pm}({\cal{F}}(\Delta_1))\otimes M^{\pm}({\cal{F}}(\Delta_2)) \\
  & \cong M({\cal{B}}_1)\otimes M^{\pm}({\cal{A}}) \\
  & \cong {\comps}[x,y,x_1,x_2,y_1,y_2,y_3,z]^{\pm}/I
\end{split}
\]
where where $I = ((y-1)(y-x), G_1, G_2)$ and $G_1$ and $G_2$ are as in
(\ref{eqn:dim-1}) and (\ref{eqn:dim-2}) respectively, that is
explicitly given by the following
\[
\begin{split}
I = \left(\right. & \left. (y-1)(y-x),\right. \\
       & \left.  (y_1 - 1)(y_1 - x_1), (y_2 - 1)(y_2 - x_2),
                (y_3 - x_1)(y_3 - x_2),\right. \\
       & \left.  (z-1)(z-y_3), (z-x_1)(z-y_2), (z-x_2)(z-y_1), \right. \\
       & \left.  (z-y_1)(z-y_2), (z-y_1)(z-y_3), (z-y_2)(z-y_3)\right.\left.\right).
\end{split}
\]

\section{Summary and questions}
\label{sec:summary}

We conclude this article by a brief summary of what has been done and
pose some relevant questions.

In Section~\ref{sec:defs-obs} we reviewed some useful properties
of Minkowski rings as ${\comps}[x_1,\ldots,x_n]/I$ and
related Laurent polynomial rings ${\comps}[x_1,\ldots,x_n]^{\pm}/I$.
Such rings capture all the algebraic relations stemming from a finite
collection of polytopes and their Minkowski sums.

Section~\ref{sec:simple} was dedicated to the explicit Minkowski
ring of a single closed interval together with its endvertices.
We showed that the Minkowski ring
$R = R_{\alpha,\beta} := M(\{\alpha\},\{\beta\},[\alpha,\beta])$
depends not only on the combinatorics of the interval $[\alpha,\beta]$
and its endpoint but also on its explicit embedding in the real line
and related linear dependences over ${\nats}$ or ${\ints}$
in Proposition~\ref{prp:4-rings} and Corollary~\ref{cor:3-rings}.
Further we proved that the types of rings in Proposition~\ref{prp:4-rings} 
are always non-isomorphic as rings. We then obtain some corollaries
for the Minkowski rings of $d$-dimensional boxes in ${\reals}^d$
in Corollaries~\ref{cor:d-box} and~\ref{cor:d-box-pm}.

Section~\ref{sec:Coxeter} was dedicated to establishing the Minkowski ring
$M^{\pm}({\cal{A}})$ of all closed polygonal sets in ${\reals}^2$
with vertices from the triangular grid
${\ints}(1,0) + {\ints}(\frac{1}{2},\frac{\sqrt{3}}{2})$ and where
each side is horizontal or has slope $\pm\sqrt{3}$. The
main result here is Theorem~\ref{thm:A} that describes completely
the algebraic structure of $M^{\pm}({\cal{A}})$. We then 
established some properties of general identities in Minkowski rings
$M({\cal{F}}(P))$ of a polytope and all its faces in
Corollary~\ref{cor:minimal} and Proposition~\ref{prp:indicator-iff}.

In the final Section~\ref{sec:Cartesian} we showed that the Minkowski
ring construction of polytopes behaves well under Cartesian products
as stated in Theorem~\ref{thm:Cart-isom}. A natural question is
therefore if Minkowski rings behave well under other constructions
as well. Recall the notion of a {\em pyramid} over a
polytope~\cite[p.~9]{Ziegler}.
\begin{question}
\label{qst:pyr}
If $P$ is a $d$-dimensional polytope and $\pyr(P)$ is the pyramid over
$P$, can the Minkowski ring $M^{\pm}({\cal{F}}(\pyr(P)))$ be described
in terms of $M^{\pm}({\cal{F}}(P))$, the Minkowski ring of $P$, in some
form? More casually, does the Minkowski ring construction behave well
under the pyramid construction?
\end{question}
One can view the pyramid construction as a special case of the
{\em fee join} $P*Q$ or just {\em join} of two polytopes
$P$ and $Q$~\cite[p.323]{Ziegler}.
\begin{question}
\label{qst:join}
For two polytopes $P$ and $Q$ can the Minkowski ring
$M^{\pm}({\cal{F}}(P*Q))$ of the free join of $P$ and $Q$ be
described in terms of the Minkowski rings
$M^{\pm}({\cal{F}}(P))$ and $M^{\pm}({\cal{F}}(Q))$?
\end{question}

\subsection*{Author Contribution:} The authors Geir and Jim jointly wrote
the manuscript text, both authors reviewed the manuscript.

\subsection*{Conflict of Interest:} None

\subsection*{Data Availability Statement:} None

\subsection*{Funding Declaration:} None

\subsection*{Acknowledgments}
Sincere thanks to ... 

\bibliographystyle{amsalpha}
\bibliography{Mr-arXiv}

\providecommand{\bysame}{\leavevmode\hbox to3em{\hrulefill}\thinspace}
\providecommand{\MR}{\relax\ifhmode\unskip\space\fi MR }
\providecommand{\MRhref}[2]{%
  \href{http://www.ams.org/mathscinet-getitem?mr=#1}{#2}
}
\providecommand{\href}[2]{#2}
\begin{thebibliography}{McM94}

\bibitem[AL24]{A-Lawrence}
Geir Agnarsson and Jim Lawrence, \emph{Power-closed ideals of polynomial and
  {L}aurent polynomial rings}, J. Pure Appl. Algebra \textbf{228} (2024),
  no.~12, Paper No. 107733, 29. \MR{4752836}

\bibitem[Alu09]{Aluffi}
Paolo Aluffi, \emph{Algebra: chapter 0}, Graduate Studies in Mathematics, vol.
  104, Amer.~Math.~Soc., Providence, RI, 2009.

\bibitem[AM69]{Atiyah-Macdonald}
Michael~Francis Atiyah and Ian~G. Macdonald, \emph{Introduction to commutative
  algebra}, Addison-Wesley, Reading, Mass.-London-Don Mills, Ont., 1969.

\bibitem[Bar02]{Barvinok}
Alexander Barvinok, \emph{A course in convexity}, Graduate Studies in
  Mathematics, vol.~54, American Mathematical Society, Providence, RI, 2002.
  \MR{1940576}

\bibitem[BR15]{Beck-and-Robins}
Matthias Beck and Sinai Robins, \emph{Computing the continuous discretely},
  second ed., Undergraduate Texts in Mathematics, Springer, New York, 2015,
  Integer-point enumeration in polyhedra, With illustrations by David Austin.
  \MR{3410115}

\bibitem[FS92]{Jay-Klaus}
Klaus~G. Fischer and Jay Shapiro, \emph{The prime ideal structure of the
  minkowski ring of polytopes}, J.~Pure Appl.~Algebra \textbf{78} (1992),
  no.~3, 239--251.

\bibitem[Gro74]{Groemer}
Helmut Groemer, \emph{On the euler characteristic in spaces with a separability
  property}, Math.~Ann. \textbf{211} (1974), 315--321.

\bibitem[Had55]{Hadwiger}
H.~Hadwiger, \emph{Eulers {C}harakteristik und kombinatorische {G}eometrie}, J.
  Reine Angew. Math. \textbf{194} (1955), 101--110. \MR{73221}

\bibitem[Law85]{Lawrence}
Jim Lawrence, \emph{Minkowski rings}, Preprint, 1985.

\bibitem[Law88]{Lawrence-(val+pol)}
\bysame, \emph{Valuations and polarity}, Discrete Comput. Geom. \textbf{3}
  (1988), no.~4, 307--324. \MR{947219}

\bibitem[Law97]{Lawrence-(Euler)}
\bysame, \emph{A short proof of {E}uler's relation for convex polytopes},
  Canad. Math. Bull. \textbf{40} (1997), no.~4, 471--474. \MR{1611351}

\bibitem[McM94]{McMullen}
Peter McMullen, \emph{Applications of the polytope algebra},
  Rend.~Circ.~Mat.~Palermo \textbf{2} (1994), no.~35, 203--216.

\bibitem[Mor93]{Morelli}
Robert Morelli, \emph{A theory of polyhedra}, Adv.~Math. \textbf{97} (1993),
  no.~1, 1--73.

\bibitem[Zie95]{Ziegler}
G\"unter~M. Ziegler, \emph{Lectures on polytopes}, Graduate Texts in
  Mathematics, vol. 152, Springer-Verlag, New York, 1995. \MR{1311028}

\end{thebibliography}

\end{document}